\documentclass[11pt]{article}
\usepackage{graphicx}
\usepackage{amsmath,amsfonts,amsthm}
\usepackage{amssymb,latexsym}
\usepackage{fixmath}
\usepackage{mathrsfs,amsbsy}
\usepackage{dsfont}
\usepackage{enumerate}

\DeclareMathOperator{\F}{F}
\DeclareMathOperator{\HH}{H}
\DeclareMathOperator{\T}{T}

\def\ba{\pmb{a}}
\def\bb{\pmb{b}}

\def\be{\pmb{e}}

\def\bm{\pmb{m}}
\def\bn{\pmb{n}}

\def\bp{\pmb{p}}

\def\bt{\pmb{t}}
\def\bu{\pmb{u}}
\def\bv{\pmb{v}}
\def\bw{\pmb{w}}

\def\bz{\pmb{z}}

\newtheorem{proposition}{Proposition}[section]
\newtheorem{theorem}{Theorem}[section]
\newtheorem{lemma}{Lemma}[section]
\newtheorem{corollary}{Corollary}[section]

\theoremstyle{definition}
\newtheorem{definition}{Definition}[section]
\newtheorem{remark}{Remark}[section]
\newtheorem{example}{Example}[section]

\numberwithin{equation}{section}
\numberwithin{figure}{section}
\numberwithin{table}{section}

\allowdisplaybreaks

\usepackage{kbordermatrix}

\usepackage{caption}
\usepackage{amsfonts,amsthm}
\usepackage{amsmath,amssymb,enumerate,color}
\usepackage{algorithm,algorithmic}
\usepackage{epsfig,graphicx,psfrag,mathrsfs,subfigure}
\usepackage{eucal}
\usepackage{yhmath}
\usepackage{hyperref}
\usepackage{diagbox}
\RequirePackage{multirow}
\usepackage{textcomp}
\usepackage[T1]{fontenc}
\usepackage{extarrows}

\usepackage{rotating}
\usepackage{lscape}
\usepackage{bm}
\usepackage{xurl}
\usepackage{xspace}
\usepackage{tikz}
\usetikzlibrary{petri} 
\usetikzlibrary{shadows,arrows,positioning,shapes.geometric} 

\usepackage{geometry}
\def\ba{\pmb{a}}
\def\bb{\pmb{b}}

\def\be{\pmb{e}}

\def\bn{\pmb{n}}

\def\bp{\pmb{p}}

\def\bt{\pmb{t}}
\def\bu{\pmb{u}}
\def\bv{\pmb{v}}
\def\bw{\pmb{w}}

\def\bz{\pmb{z}}
\def\bzs{\mathbf{0}}

\def\scrR{\mathscr{R}}

\def\bbC{\mathbb{C}}

\def\bbP{\mathbb{P}}

\renewcommand{\algorithmicrequire}{\textbf{Input:}}
\renewcommand{\algorithmicensure}{\textbf{Output:}}

\numberwithin{equation}{section}
\numberwithin{figure}{section}
\numberwithin{table}{section}

\graphicspath{{figures/}{figure/}{pictures/}%
	{picture/}{pic/}{pics/}{image/}{images/}}
	
\title{{\sf NEP\_MiniMax}: An approach for NEPs based on minimax approximation of the split  {coefficient} vector}

 \author{Chenkun Zhang\thanks{School of
  Mathematical Sciences, Soochow University, Suzhou 215006, Jiangsu, China. Email: {\tt 19816896524@163.com.}}  \and Jiawei Gu\thanks{School of Mathematical Sciences, Soochow University, Suzhou 215006, Jiangsu, China. Email: {\tt 20244207049@stu.suda.edu.cn}.}\and Lei-Hong Zhang\thanks{Corresponding author. School of Mathematical Sciences, Soochow University, Suzhou 215006, Jiangsu, China. This work was
 supported in part by the National Natural Science Foundation of China (NSFC-12471356, NSFC-12371380).
        Email: {\tt longzlh@suda.edu.cn}.} }
 \date{\today}

\begin{document}
\maketitle
\begin{abstract}
We propose {\sf NEP\_MiniMax}, a novel computational method for solving nonlinear eigenvalue problems (NEPs) $T(\lambda)\bu = \bzs$ on a   Jordan domain $\Omega \subset \mathbb{C}$.  {For an NEP in split form $T(x)=\sum_{i=1}^s t_i(x)E_i$, the method applies the {\sf m-d-Lawson} algorithm to the split coefficient vector-valued function $\bt(x)=[t_1(x),\ldots,t_s(x)]^{\rm T}$ on $\Omega$, rather than directly minimizing an error for the matrix-valued function $T$. The resulting rational minimax approximant ${\boldsymbol{\xi}}^*(x)=[r_1^*(x),\ldots,r_s^*(x)]^{\rm T}$ induces the rational matrix surrogate $R^*(x)=\sum_{i=1}^s r_i^*(x)E_i=P^*(x)/q^*(x)\approx T(x)$.} The method combines  {this coefficient-vector approximation} with a structure-exploiting linearization technique.  {A matrix error bound transfers the uniform coefficient-vector approximation accuracy to $R^*$ and yields computable eigenpair residual estimates.} Eigenpairs are then computed by solving a polynomial eigenvalue problem $P^*(\lambda) \bu = \bzs$ via a strong linearization that exactly preserves eigenvalue multiplicities  {of $P^*$}. Numerical experiments on some problems from the NLEVP collection demonstrate competitiveness with state-of-the-art methods (e.g., Beyn, NLEIGS, SV-AAA) in efficiency and accuracy, with theoretical error bounds directly relating eigenpair  {residuals} to the  {split coefficient-vector approximation} quality.
\end{abstract}

\medskip
{\small
\medskip

} 
 

\section{Introduction}\label{sec:introduction}
Let $\Omega \subset \mathbb{C}$  be a   Jordan domain whose boundary $\Gamma:=\partial \Omega$ is a simple Jordan curve. Let $\bar{\Omega}:={\Omega}\cup \Gamma$. Given a continuous matrix-valued function $T:  \bar{\Omega} \rightarrow \mathbb{C}^{n \times n}$ that is analytic in $\Omega$ (i.e., each matrix entry $T_{ij}(x)$ is continuous on $\bar\Omega$ and analytic in $\Omega$), we consider the nonlinear eigenvalue problem (NEP) of finding $(\lambda, \pmb{u}) \in \Omega \times (\mathbb{C}^n \setminus \{\bzs\})$ such that
\begin{equation} \label{eq:nep}
T(\lambda) \pmb{u} = \bzs.
\end{equation}
Here, $\lambda$ is an eigenvalue of $T$ and $\pmb{u}$ is the corresponding right eigenvector;  $(\lambda, \pmb{u})$ is a right eigenpair of $T(x)$. The spectrum of $T$ in $\Omega$, denoted $\Lambda(T)$, is the set of all eigenvalues lying in $\Omega$.
A left eigenpair $(\lambda, \pmb{\tilde u})$ is the one satisfying $\pmb{\tilde u}^{\HH} T(\lambda) = \bzs^{\HH}$ for $\pmb{\tilde u}\in \mathbb{C}^n \setminus \{\bzs\}$. Since left eigenpairs of $T(x)$ can be derived from right eigenpairs of  its conjugate transpose  $T(x)^{\HH}$, we focus   on computing right eigenpairs. Adopting the general setting for NEPs presented in \cite{guti:2017}, we assume $T(x) $ admits a split form
\begin{equation} \label{eq:split_form_T}
T(x) = \sum_{i=1}^s t_i(x) E_i,
\end{equation}
where $t_i: \bar\Omega \rightarrow \mathbb{C}$ is  analytic in $\Omega$ and continuous up to the boundary    $\Gamma$, and $E_i \in \mathbb{C}^{n \times n}$ is a  constant matrix for $i = 1, \dots, s$.

The NEP in the form of \eqref{eq:split_form_T} arises in numerous computational science and engineering disciplines, such as acoustics, control theory, fluid mechanics, and structural mechanics. For a thorough review of applications and state-of-the-art numerical methods as of 2007, we refer readers to \cite{guti:2017}. Furthermore, in 2013,  \cite{behm:2013} provides a MATLAB-based benchmark collection NLEVP that currently contains approximately $80$ test problems from various application domains. This comprehensive library encompasses polynomial eigenvalue problems, rational eigenvalue problems, as well as more general nonlinear eigenvalue problems.
Owing to its unified interface for diverse NEPs and customizable problem generation parameters, NLEVP has emerged as an essential resource for algorithm development and numerical validation in the field of NEP.

Numerical methods for  NEPs can be broadly categorized into four classes: linearization-based methods, iterative projection methods, contour integral-based methods, and Newton-type methods. Linearization-based methods, e.g.,  \cite{limp:2022,saem:2020,bemm:2013,bemm:2015}, begins with the fact that polynomial \cite{mamm:2006,mamm:2006b} and rational NEPs \cite{subz:2011} can often be transformed into generalized eigenvalue problems (GEPs). For general NEPs of the form \eqref{eq:split_form_T}, rational or polynomial approximation techniques (including interpolation and numerical approximation) applied to the nonlinear terms $t_i(x)$ yield tractable polynomial/rational eigenvalue problems. Iterative projection methods, including nonlinear Arnoldi and Jacobi-Davidson approaches (e.g., \cite{tasa:2024,voss:2007}), extend classical linear eigensolvers to nonlinear eigenvalue problems by projecting the NEP onto carefully constructed subspaces, and are   effective for large-scale problems due to their use of adaptive restarting and repeated linear solves.  Contour integral-based methods such as those in \cite{asst:2009,beyn:2012,breg:2023,poke:2015,gamp:2018,lirs:2025,lirs:2026,sasu:2003}, exemplified by Beyn's method and the FEAST framework, employ complex analysis to extract eigenvalues within specified regions via contour integration, avoiding full linearization while targeting interior eigenvalues efficiently. Newton-type methods \cite{gabl:2016,kres:2009,kubl:1969,lanc:1961,lanc:2002,neum:1985,schr:2008} provide locally superlinear convergence through iterative refinement of eigenvalue-eigenvector pairs, particularly effective for smooth nonlinearities.

{In this paper, we develop a linearization-based approach for eigenvalues in $\Omega$ whose approximation stage is centered on the split coefficient vector.} Our proposed method employs  {recent} developments on rational minimax approximations for vector-valued (or matrix-valued) functions \cite{zhzz:2025,zhzh:2026b} and the related {\sf m-d-Lawson} iteration. According to the  {split} form \eqref{eq:split_form_T}, we decouple the treatment of analytic functions $t_i(x)$ from constant matrices $E_i$ by computing an ideal rational  {approximant for the split coefficient vector}, \({\boldsymbol{\xi}}^{\rm best}(x)\), through 
\begin{equation}\label{eq:bestf0}
    {\boldsymbol{\xi}}^{\rm best}(x):=[r^{\rm best}_1(x),\dots,r^{\rm best}_s(x)]^{\rm T} :=\arg\inf_{{\boldsymbol{\xi}} \in \mathscr{R}_s}\sup_{x \in\bar \Omega}\left\|[t_1(x),\dots,t_s(x)]^{\rm T}-{\boldsymbol{\xi}}(x)\right\|_2,
\end{equation}
and then construct a rational NEP problem  with  \(R^{\rm best}(x)=\sum\limits_{i=1}^{s}r^{\rm best}_i(x)E_i\approx T(x)\), 
where 
\begin{equation}\label{eq:rats}
\scrR_{s}:=\left\{{\boldsymbol{\xi}}(x)=\left[\begin{array}{ccc}r_{1}(x)     \\\ \vdots  \\ r_{s}(x)  \end{array}\right]  \Big| r_{i}(x)=\frac{p_{i}(x)}{q(x)}, p_{i}\in \bbP_{n_{i}},~0\not\equiv q\in\bbP_{d}\right\},
\end{equation} 
$n_{i}\ge 0$ and $d\ge 0$ are prescribed integers. The $i$th function $r_i(x)=\frac{p_{i}(x)}{q(x)}$ of ${\boldsymbol{\xi}}(x)$ is said to be a rational function of type $(n_{i},d)$, where $p_{i}\in \bbP_{n_{i}}$ is the $i$th entry of a vector-valued polynomial $\bp:\bbC\rightarrow \bbC^{s}$.   

It is worth mentioning that the formulation of \eqref{eq:bestf0}  exhibits two notable features for solving NEPs. First, since all component rational approximations $r_i(x)$ of $\bm{\xi}^{\text{best}}(x)$ (when it exists)  share a common denominator $q(x)$, the resulting rational eigenvalue problem with $R^{\rm best}(x)$ can be converted to a polynomial eigenvalue problem. Second, this shared denominator structure implies that the solution $\bm{\xi}^{\text{best}}(x)$ of \eqref{eq:bestf0} behaves fundamentally differently from componentwise rational minimax approximations obtained by independently minimizing the supremum error of each component $t_{i}(x)$ ($1 \le i \le s$) on $\bar\Omega$. This distinction becomes particularly clear in the polynomial case (when $d = 0$ and $q \equiv 1$), where the joint approximation approach in \eqref{eq:bestf0} is   superior to independent componentwise approximation because
\begin{align*}
&\inf_{{\boldsymbol{\xi}} \in \mathscr{R}_s}\sup_{x \in\bar \Omega}\left\|[t_1(x),\dots,t_s(x)]^{\rm T}-{\boldsymbol{\xi}}(x)\right\|_2^2= \inf_{{p_{i}\in \bbP_{n_i}}} \sup_{x\in \bar \Omega} \sum_{i=1}^s   |t_{i}(x) - p_{i}(x)|^2 \\
\leq&\inf_{{p_{i}\in \bbP_{n_i}}}  \sum_{i=1}^s\sup_{x\in \bar \Omega}   |t_{i}(x) - p_{i}(x)|^2 =   \sum_{i=1}^s\left(\inf_{{p_{i}\in \bbP_{n_i}}}\sup_{x\in \bar \Omega}   |t_{i}(x) - p_{i}(x)|^2\right),
\end{align*}
demonstrating that the unified formulation in \eqref{eq:bestf0}   improves upon decoupled componentwise minimax error over $\bar\Omega$.

For the existence of $ {\boldsymbol{\xi}}^{\rm best}(x)$ in \eqref{eq:bestf0}, the fundamental existence of the rational minimax approximant in the classical scalar-valued case on a set dense in itself, established by Walsh \cite{wals:1931,wals:1969}, has recently been extended to  {vector-valued (and, more generally, matrix-valued) rational minimax approximation  \cite{zhzh:2026b};   here it is applied to the split coefficient-vector problem} \eqref{eq:bestf0} in \cite[Theorem 2.1]{zhzh:2026b}. In particular, the optimal approximant ${\boldsymbol{\xi}}^{\rm best}(x)$ is therefore guaranteed to exist for any $\scrR_{s}$ with given $d\ge 0$ and $n_{i}\ge 0~(1\le i\le s)$.

Computing the exact ${\boldsymbol{\xi}}^{\rm best}(x)$  is not an easy numerical task. Fortunately, since each entry $t_{i}(x)$ is analytic in its interior and continuous  up to the boundary $\partial\Omega$, the maximum Frobenius norm principle holds; that is, 
$$\sup_{x \in\bar \Omega}\left\|\bt(x)-{\boldsymbol{\xi}}(x)\right\|_2$$
is attained on   $\partial \Omega$ for any ${\boldsymbol{\xi}}(x)\in \scrR_{s}$ with poles outside $\bar\Omega$ (see Theorem 1 in \cite{cond:2020}), 
where 
\begin{equation}\label{eq:def_Fx}
    \bt(x)=\begin{bmatrix}
        t_1(x)\\ \vdots \\ t_s(x)
    \end{bmatrix}: \bar\Omega \rightarrow \mathbb{C}^{s}.
\end{equation}
Consequently, solving   \eqref{eq:bestf0} with carefully chosen boundary nodes $\mathcal{X}=\{x_j\}_{j=1}^m \subset \partial \Omega$ with $m\ge \max_{1\le i\le s}(n_i+d+2)$ turns out to be an efficient numerical approach for computing minimax rational approximants on $\bar\Omega$. This is realized by the {\sf m-d-Lawson} iteration \cite{zhzz:2025}. Specifically, we compute an approximant  \({\boldsymbol{\xi}}^* \approx {\boldsymbol{\xi}}^{\rm best}\) from the discrete rational minimax approximation 
\begin{equation}\label{eq:bestfX}
     {\boldsymbol{\xi}}^*(x):=[r^*_1(x),\dots,r^*_s(x)]^{\rm T} :=\arg\min_{{\boldsymbol{\xi}} \in \mathscr{R}_s}\sup_{x \in {\cal X}}\left\|\bt(x)-{\boldsymbol{\xi}}(x)\right\|_2,
\end{equation}
and form 
\begin{equation}\label{eq:RT_def}
R^*(x)=\sum\limits_{i=1}^{s}r^*_i(x)E_i\approx R^{\rm best}(x),
\end{equation} 
where $r_i^*(x)=p^*_{i}(x)/q^*(x)$.  {Thus the minimax property applies to the split coefficient vector, while \(R^*(x)\) is the induced rational matrix surrogate whose eigenpairs are used to approximate those of \(T(x)\).}

In theory, it is pointed out (see \cite[Section 4.2]{zhzh:2026b}) that one may not necessarily recover exactly the ideal ${\boldsymbol{\xi}}^{\rm best}$ from ${\boldsymbol{\xi}}^*$, and a sufficient condition on the {\it extreme points}, i.e., points $z\in\bar\Omega$ satisfying $$\|{\boldsymbol{\xi}}^{\rm best}(z)-\bt(z)\|_{2}=\sup_{x\in \bar\Omega} \|{\boldsymbol{\xi}}^{\rm best}(x)-\bt(x)\|_{2},$$ associated with ${\boldsymbol{\xi}}^{\rm best}$ is provided in \cite[Theorem 4.2]{zhzh:2026b}, ensuring that a properly chosen boundary point set ${\cal X}$ in \eqref{eq:bestfX} can recover ${\boldsymbol{\xi}}^{\rm best}$. 
Our numerical experiments  also  show that,  in most cases of the sampling ${\cal X}$,  the  {approximate} rational minimax approximant \({\boldsymbol{\xi}}^*\), computed via {\sf m-d-Lawson},   remains pole-free in the region of interest. 
Setting \(R^*=R^*(x) = \frac{P^*(x)}{q^*(x)} := \frac{1}{q^*(x)}\sum_{i=1}^s p_i^*(x) E_i\), we consequently have \(\Omega \cap \Lambda(R^*) = \Omega \cap \Lambda(P^*)\). This equivalence allows us to compute the eigenvalues of the rational matrix function \(R^*(x)\) by solving the eigenvalue problem for the matrix polynomial \(P^*(x)\) restricted to $\Omega$. For small to moderate dimension $n$, we employ linearization followed by the QZ algorithm, while for large-scale problems, we utilize filtered subspace iterations \cite{chli:2025} or the   compact rational Krylov   (CORK)   framework \cite{bemm:2015}. 
 {The complete framework of our coefficient-vector minimax based NEP solution approach is illustrated in Figure~\ref{fig:framework_nep_solve}.}

\begin{figure}[htbp]
    \centering
    \includegraphics[width=0.92\textwidth]{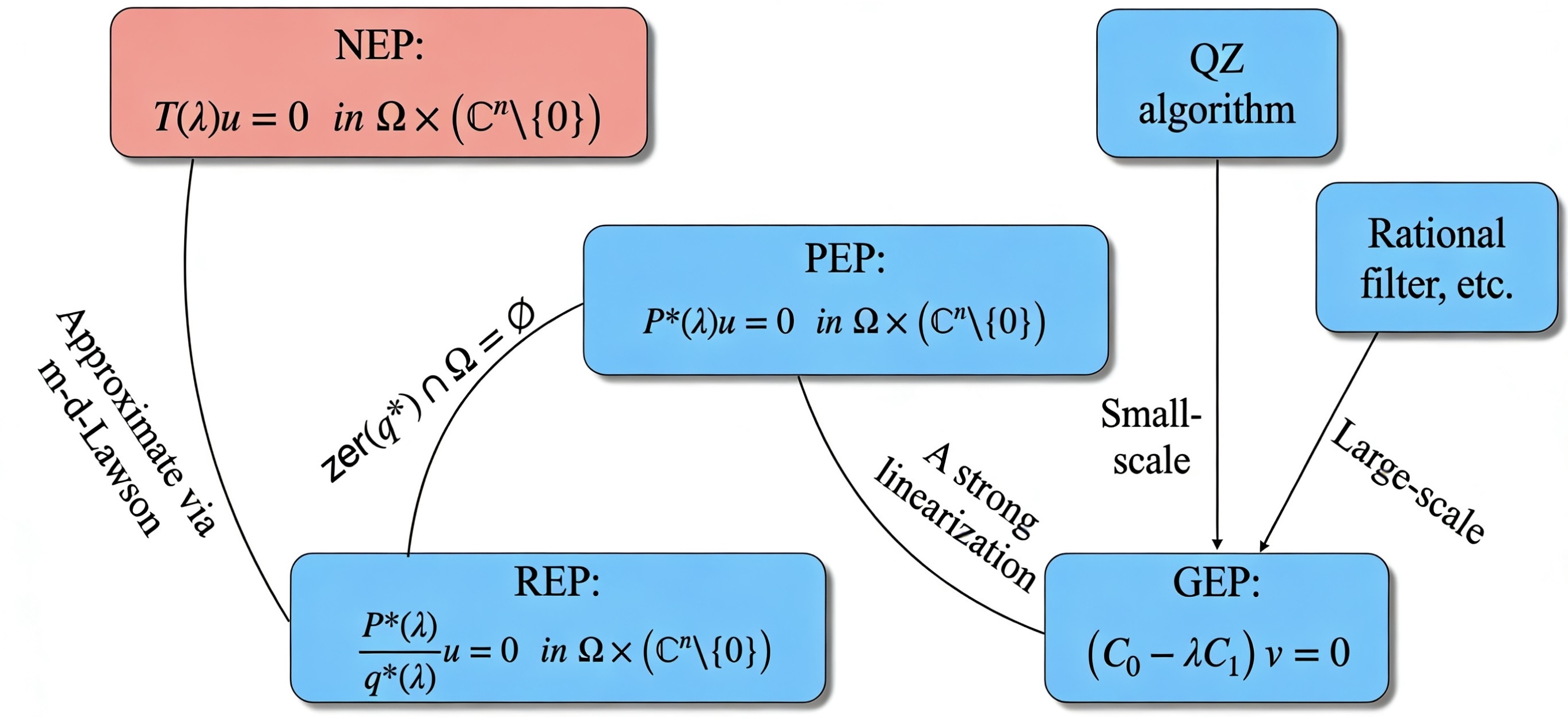}
    \caption{\footnotesize  {Framework of {\sf NEP\_MiniMax} for solving NEPs based on minimax approximation of the split coefficient vector via {\sf m-d-Lawson} and subsequent linearization.}}
    \label{fig:framework_nep_solve}
\end{figure}

{\it Paper structure}. Building upon the framework illustrated in Figure~\ref{fig:framework_nep_solve}, this paper is organized as follows: Section~\ref{sec:mdlawson} introduces the rational minimax approximation problem~\eqref{eq:bestfX}  {for the split coefficient vector}, presenting its dual formulation and the dual-based {\sf m-d-Lawson} iteration~\cite{zhzz:2025}; Section~\ref{sec:erroranalysis} establishes error bounds for $\sup_{x\in\bar\Omega}\|T(x)-R^*(x)\|_{2}$  {induced by the coefficient-vector approximation}, providing theoretical justification for our {\sf NEP\_MiniMax} method; Section~\ref{sec:linearization} discusses linearization techniques for the rational eigenvalue problem $R^*(\lambda)\bu=\bzs$; Section~\ref{sec:largescale} addresses large-scale computation through rational filtering, shift-invert strategies, and the CORK framework~\cite{bemm:2015}; Section~\ref{sec:numerical} validates our approach using NLEVP benchmark problems, demonstrating competitive performance against  {Beyn's} method \cite{beyn:2012}, NLEIGS and SV-AAA \cite{limp:2022} in efficiency and accuracy;   concluding remarks are given in Section~\ref{sec:conclusion}.

{\it Notation}. Throughout this paper, ${\rm i}=\sqrt{-1}$ denotes the imaginary unit and $\mathbb{P}_n$ contains complex polynomials of degree $\leq n$. 
 Vectors appear in bold lowercase, while $\mathbb{C}^{n \times m}$ (resp.\ $\mathbb{R}^{n \times m}$) denotes all $n \times m$ complex (resp.\ real) matrices, with $I_{n} \equiv [\bm{e}_1,\dots,\bm{e}_n] \in \mathbb{R}^{n \times n}$ being the identity matrix ($\bm{e}_i$ is its $i$th column for $1 \le i \le n$). The operator ${\rm diag}(\bm{x})={\rm diag}(x_1,\dots,x_m)$ creates diagonal matrices from $\bm{x} \in \mathbb{C}^m$, and $\bm{x}./\bm{y}=[x_1/y_1,\dots,x_m/y_m]^{\T}$ defines element-wise division ($\bm{x},\bm{y} \in \mathbb{C}^m$, $y_j \ne 0$). For $A \in \mathbb{C}^{m \times n}$, $A^{\rm H}$ (resp.\ $A^{\rm T}$) is conjugate (resp.\ regular) transpose,  ${\rm span}(A)$ the column space, $\|\cdot\|_{\rm F}$ the Frobenius norm, and ${\rm tr}(\cdot)$ the trace. 
Following MATLAB notation, $A(\mathcal{J}_1, \mathcal{J}_2)$ denotes the submatrix of $A$ formed by rows indexed by $\mathcal{J}_1 \subseteq [m]=:\{1,\dots,m\}$ and columns indexed by $\mathcal{J}_2 \subseteq  [n]$. Finally, ${\rm zer}(f)$ denotes zeros of the function $f$.
  
\section{Computing a rational minimax approximant $\xi^*$}\label{sec:mdlawson}
We begin with a concise overview of the {\sf m-d-Lawson} algorithm \cite{zhzz:2025} and its application  {to the split coefficient vector $\bt(x)=[t_1(x),\ldots,t_s(x)]^{\rm T}$ associated with $T(x)$ in \eqref{eq:split_form_T}}. The implementation details follow essentially from the theoretical framework developed in \cite{zhzz:2025}.
 
Given $ \mathscr{R}_{s}$ in \eqref{eq:rats} with prescribed non-negative integers \(n_{i}\) and \(d\), the minimax problem \eqref{eq:bestf0} is an extension of the classical problem in approximation theory \cite{tref:2019a,wals:1931,wals:1969} for  {complex scalar-valued} functions. As mentioned in Section \ref{sec:introduction}, by the analyticity of each  {split coefficient function} $t_i(x)$ in $\Omega$ and the continuity on $\partial\Omega$, the maximum norm principle  \cite[Theorem 1]{cond:2020} applies:
$$\sup_{x \in \partial\Omega} \left\| \bt(x)- {\boldsymbol{\xi}}(x) \right\|_2 = \sup_{x \in \bar\Omega} \left\|\bt(x) - {\boldsymbol{\xi}}(x) \right\|_2,$$
where ${\boldsymbol{\xi}}(x) \in \scrR_s$ is any rational approximation with poles exterior to $\Omega$. Thus, we resort to numerical methods for its discretized problem by sampling nodes ${\cal X}=\{x_j\}_{j=1}^m$ (\(m \ge \max_{1 \le i \le s}\{n_{i}+d+2\}\)) in\footnote{It is known (see e.g., \cite{this:1993, zhan:2026}) that the rational minimax approximant  in \eqref{eq:rats} of $\bt(x)$ on the Jordan contour $\Gamma=\partial \Omega$ does not necessarily coincide with $\boldsymbol{\xi}^{\rm best}$ on the Jordan domain  $\Omega$; thus, including interior samples from $\Omega$ in ${\cal X}$  may be helpful for $\boldsymbol{\xi}^{*}$ to approximate $\boldsymbol{\xi}^{\rm best}$ numerically.} $\bar\Omega$, and consider the discrete rational minimax approximation \eqref{eq:bestfX} for the vector-valued function $\bt(x)$.

To be precise, denote by \(\{(x_{\ell},\bt(x_{\ell}))\}_{\ell=1}^m\) the sampled data 
and consider  
\begin{equation}\label{eq:minimax}
    \eta_{\infty}:=\inf_{{\boldsymbol{\xi}} \in \mathscr{R}_{s}}\|\bt(x)-{\boldsymbol{\xi}}(x)\|^2_{\infty, {\cal X}},
\end{equation}
where 
\begin{equation}\nonumber
    \|\bt(x)-{\boldsymbol{\xi}}(x)\|_{\infty, {\cal X}}:=\sup_{1 \le \ell \le m}\|\bt(x_{\ell})-{\boldsymbol{\xi}}(x_{\ell})\|_{2}.
\end{equation}
Following the traditional treatments of rational approximations for vector-valued or matrix-valued functions  \cite{begu:2015, begu:2017,gogu:2021, gust:2009, limp:2022,zhzy:2025}, we enforce a common denominator $q \in \bbP_d$ for all entries $r_{i}(x)$ of ${\boldsymbol{\xi}}$ in \eqref{eq:bestfX}.  In case when the infimum of \eqref{eq:minimax} is attainable, we call the function \({\boldsymbol{\xi}}^*=\frac{1}{q^*}\bp^*\in \mathscr{R}_{s}\) from 
\begin{equation}\nonumber
    {\boldsymbol{\xi}}^* \in \arg\min_{{\boldsymbol{\xi}} \in \mathscr{R}_{s}}\|\bt(x)-{\boldsymbol{\xi}}(x)\|_{\infty, {\cal X}}^2
\end{equation} 
the rational minimax approximant \cite{tref:2019a} of \(\bt(x)\) over \({\cal X}\). We refer to \cite{zhzh:2026b} for the optimality conditions,  relations with the dual-based methods as well as a sufficient condition ensuring that a properly chosen boundary point set ${\cal X}$ can recover ${\boldsymbol{\xi}}^{\rm best}(x)$ of \eqref{eq:bestf0}. 
\begin{remark}\label{rmk:entry_approximate}
We approximate the vector-valued function $\bt(x)$ constructed from the components $t_i(x)$ of the split form of $T(x)$, rather than $T(x)$ itself as those in the contour  {integral-based} methods \cite{asst:2009,beyn:2012,breg:2023,poke:2015,gamp:2018,lirs:2025,sasu:2003}, primarily for computational efficiency. In practice, the dimension parameter $n$ in \eqref{eq:split_form_T} typically dominates $s$, and the matrices $E_i$ are often large-scale and sparse. This approach yields significant reductions in both computational complexity and storage demands compared to direct min-max approximation of $T(x)$. A detailed error analysis comparing the rational approximation $R^*(x)= \sum_{i=1}^{s} r_i^*(x) E_i$ with $T(x)$ is provided in Section \ref{sec:erroranalysis}.
\end{remark}

\subsection{A dual problem of \eqref{eq:bestfX}}\label{subsec:dual}
Following  \cite{zhyy:2025,zhzz:2025}, to represent a rational function \(r_{i}(x)=\frac{p_{i}(x)}{q(x)}\), we choose  {basis functions} \(\{\psi_0(x),\dots,\psi_{n_{i}}(x)\}\) and \(\{\phi_0(x),\dots,\phi_{d}(x)\}\) to  {parameterize}
\begin{equation}\label{eq:paris}
    r_{i}(x)=\frac{p_{i}(x)}{q(x)}=\frac{[\psi_0(x),\dots,\psi_{n_{i}}(x)]\pmb{a}_{i}}{[\phi_0(x),\dots,\phi_d(x)]\pmb{b}},\ {\rm for\ some}\ \pmb{a}_{i} \in \mathbb{C}^{n_{i}+1},\ \pmb{b} \in \mathbb{C}^{d+1}.
\end{equation}
The monomial basis \(\psi_i(x)=\phi_i(x)=x^i\) can usually be used. For the given \({\cal X}=\{x_{\ell}\}_{\ell=1}^m\), we denote by
\begin{equation}\nonumber
    \pmb{\Psi}_{i}=\pmb{\Psi}_{i}(x_1,\dots,x_m;n_{i}):=\begin{bmatrix}
        \psi_0(x_1)&\psi_1(x_1)&\dots&\psi_{n_{i}}(x_1)\\ \psi_0(x_2)&\psi_1(x_2)&\dots&\psi_{n_{i}}(x_2)\\ \vdots&\vdots&\ddots&\vdots\\\psi_0(x_m)&\psi_1(x_m)&\dots&\psi_{n_{i}}(x_m)
    \end{bmatrix} \in \mathbb{C}^{m \times (n_{i}+1)},
\end{equation}
the associated generalized Vandermonde matrix; therefore,
\begin{equation}\nonumber
    \pmb{p}_{i}:=[p_{i}(x_1),\dots,p_{i}(x_m)]^{\T}=\pmb{\Psi}_{i}\pmb{a}_{i} \in \mathbb{C}^m.
\end{equation}
Similarly, we denote by \(\pmb{\Phi}=\pmb{\Phi}(x_1,\dots,x_m;d) \in \mathbb{C}^{m \times (d+1)}\) the associated Vandermonde matrix for the denominator \(q(x)\) at \({\cal X}\), whose \((i,j)\)th entry is \(\phi_{j-1}(x_i)\); thus
\begin{equation}\nonumber
    \pmb{q}:=[q(x_1),\dots,q(x_m)]^{\T}=\pmb{\Phi}\pmb{b} \in \mathbb{C}^m.
\end{equation}
Let \({\boldsymbol{\xi}}(x)=[r_1(x),\dots,r_s(x)]^{\T}  \in \mathscr{R}_{s}\) with each \(r_{i}(x)=p_{i}(x)/q(x)\) irreducible. If \(\|{\boldsymbol{\xi}}(x)\|_{2}\) is bounded for any \(x \in {\cal X}\), then it is easy to see that \(q(x) \ne 0\) for any \(x \in {\cal X}\). Associated with \({\boldsymbol{\xi}}(x)\) is the maximum error
\begin{equation}\label{eq:e_xi}
    e({\boldsymbol{\xi}}):=\max_{x_{\ell} \in {\cal X}}\|\bt(x_{\ell})-{\boldsymbol{\xi}}(x_{\ell})\|_{2}^2=\|\bt(x)-{\boldsymbol{\xi}}(x)\|_{\infty,{\cal X}}^2.
\end{equation}

In \cite{yazz:2023, zhha:2025,zhyy:2025,zhzz:2025,zhzh:2026}, by introducing a real variable \(\eta\), the original minimax problem \eqref{eq:minimax} is transformed into the following linearization
\begin{align}\nonumber
&\eta_{2}:=\inf_{\eta\in \mathbb{R},~p_{i}\in \mathbb{P}_{n_{i}},~0\not \equiv q\in \mathbb{P}_{d}}\eta \\\label{eq:linearity}
 s.t., ~&\sum_{i=1}^s\left|t_{i}(x_\ell)q(x_{\ell})-p_{i}(x_{\ell})\right|^2\le \eta |q(x_{\ell})|^2, \ \ \ \forall \ell \in \left[m\right].
\end{align}
Unlike the original min-max problem \eqref{eq:minimax}, \eqref{eq:linearity} is a single-level minimization \cite{zhzz:2025}.

The development of the dual-based method {\sf m-d-Lawson} \cite{zhzz:2025} for \eqref{eq:bestfX} was primarily motivated by the need to solve \eqref{eq:linearity}. To this end,  
a dual function associated with \eqref{eq:linearity} is defined by (see \cite{zhzz:2025})
\begin{align}\nonumber
    d(\pmb{w})&=\min_{\pmb{a}_{i} \in \mathbb{C}^{n_{i}+1},\pmb{b}\in\mathbb{C}^{d+1},\sum\limits_{\ell=1}^mw_{\ell}|q(x_{\ell};\pmb{b})|^2=1}\sum\limits_{\ell=1}^{m}w_{\ell}\left(\sum\limits_{i=1}^{s}\left|t_{i}(x_{\ell})q(x_{\ell};\pmb{b})-p_{i}(x_{\ell};\pmb{a}_{i})\right|^2\right) \\ \label{eq:dual_question} &=\min_{\pmb{a} \in \mathbb{C}^{N}, \pmb{b} \in \mathbb{C}^{d+1}, \|\sqrt{W}\pmb{\Phi}\pmb{b}\|_2=1}\left\|\sqrt{\pmb{W}_{\otimes}}[-\pmb{\Theta},\pmb{F}\pmb{\Phi}]\begin{bmatrix}
        \pmb{a}\\ \pmb{b}
    \end{bmatrix}\right\|_2^2,
\end{align}
where \(N=\sum\limits_{i=1}^{s}(n_{i}+1),\ F_{i}={\rm diag}(t_{i}(x_1),\dots,t_{i}(x_m)) \in \mathbb{C}^{m \times m}\), 
\begin{equation}\label{eq:Theta_def}
    \pmb{F}=\begin{bmatrix}
        F_{1}\\ \vdots\\F_{s}
    \end{bmatrix} \in \mathbb{C}^{ms \times m},\ \ \ \pmb{\Theta}=\begin{bmatrix}
        \pmb{\Psi}_{1}&&\\&\ddots&\\&&\pmb{\Psi}_{s}
    \end{bmatrix} \in \mathbb{C}^{ms \times N},
\end{equation}
\begin{equation}\label{eq:a_def}
    \pmb{a}=[\pmb{a}_{1}^{\rm T},\dots,\pmb{a}_{s}^{\rm T}]^{\rm T} \in \mathbb{C}^N,
\end{equation}
and 
\begin{equation}\label{eq:W_def}
    \pmb{W}_{\otimes}=I_{s} \otimes W \in \mathbb{C}^{ms \times ms},\ \ \ W={\rm diag}(\pmb{w}) \in \mathbb{C}^{m \times m}.
\end{equation}
With the probability simplex \({\cal S}:=\left\{\bm{w}=[w_1,\dots,w_m]^{\rm T} \in \mathbb{R}^m:\bm{w} \ge 0, \pmb{w}^{\rm T}\pmb{e}=1 \right\}\) where \(\pmb{e}=[1,\dots,1]^{\rm T} \in \mathbb{R}^m\), the following weak duality \cite{zhzz:2025} holds:
\begin{equation}\label{eq:weak_dual}
    \max_{\pmb{w}\in{\cal S}}d(\pmb{w})\le\eta_{\infty}.
\end{equation} 
Moreover, by relying on a generalization  \cite[Theorem 4.1]{zhzz:2025}  of Ruttan's sufficient condition \cite{rutt:1985}, we can also have the following theoretical guarantee for solving the minimax approximant \({\boldsymbol{\xi}}^*(x)\).
\begin{proposition}{\rm (\cite[Theorem 4.1]{zhzz:2025})}
    Given \(m \ge \max_{1 \le i \le s}\{n_i+d+2\}\) distinct nodes \({\cal X}=\{x_{\ell}\}_{\ell=1}^m\) in \(\bar\Omega\), we have    weak duality \eqref{eq:weak_dual}. Let \(\pmb{w}^*\) be the solution to the dual problem
    \begin{equation}\label{eq:dual_question_simple}
        \max_{\pmb{w} \in {\cal S}}d(\pmb{w}),
    \end{equation}
    where \(d(\pmb{w})\) is  given in \eqref{eq:dual_question}, and \((\pmb{a}^*,\pmb{b}^*)\) be the associated solution of \eqref{eq:dual_question} that achieves the minimum \(d(\pmb{w}^*)\), where \(\pmb{a}^*\) contains the coefficient vectors \(\pmb{a}_{i}^* \in \mathbb{C}^{n_{i}+1}\) partitioned as in \eqref{eq:a_def}. Denote
       $ {\boldsymbol{\xi}}^*(x)=[r_1^*(x),\dots,r_s^*(x)]^{\T}, ~r^*_{i}(x)=\frac{p^*_{i}(x)}{q^*(x)}$
    with \(p^*_{i}(x)=[\psi_0(x),\dots,\psi_{n_{i}}(x)]\pmb{a}^*_{i}\), \(q^*(x)=[\phi_0(x),\dots,\phi_d(x)]\pmb{b}^*\). Suppose \(q^*(x_{\ell}) \ne 0\) for all \(x_{\ell} \in {\cal X}\). Then if 
    \begin{equation}\label{eq:ruttan_generalization}
        d(\pmb{w}^*)=\max_{x_{\ell} \in {\cal X}}\|\bt(x_{\ell})-{\boldsymbol{\xi}}^*(x_{\ell})\|_{2}^2:=e(\boldsymbol{\xi}^*),
    \end{equation}
    then strong duality holds, i.e., \(d(\pmb{w}^*)=\eta_{\infty}\), and \({\boldsymbol{\xi}}^*(x)\) is the global solution to \eqref{eq:minimax}. Moreover, when \eqref{eq:ruttan_generalization} holds,  the following complementary slackness property holds:
    \begin{equation}\label{eq:slackness}
        w^*_{\ell}\left(e({\boldsymbol{\xi}}^*)-\|\bt(x_{\ell})-{\boldsymbol{\xi}}^*(x_{\ell})\|_{2}^2\right)=0,\ \ \forall 1 \le \ell \le m.
    \end{equation}
\end{proposition}
Within the dual   framework, we quantify the approximation  accuracy of  ${\boldsymbol{\xi}}$ at $\bw\in{\cal S}$ corresponding to minimization \eqref{eq:dual_question} through the relative error metric when $e({\boldsymbol{\xi}})\ne 0$:
\begin{equation}\nonumber
\epsilon(\pmb{w}) := \left|\frac{d(\pmb{w}) - e({\boldsymbol{\xi}})}{e({\boldsymbol{\xi}})}\right|,
\end{equation}
where $\pmb{w}$ represents an approximation to the optimal solution $\pmb{w}^*$ of \eqref{eq:dual_question_simple}. This error measure naturally serves as the termination criterion for the vector-valued dual Lawson's iteration (denoted as {\sf m-d-Lawson}) outlined in Algorithm \ref{alg:Lawson}.

To compute the dual function \(d(\pmb{w})\), a minimization problem \eqref{eq:dual_question} needs to be solved. The following proposition provides the optimality condition for \eqref{eq:dual_question}.  In particular, \eqref{eq:item3_1} connects to Kolmogorov dual criteria \cite[Section 5.3]{zhzh:2026b}, while item (iii) provides efficient ways to compute $d(\bw)$ and the associated $\ba,\bb$  in \eqref{eq:dual_question}.

\begin{proposition}{\rm {(\cite[Proposition 3.1]{zhzz:2025})}}\label{prop:item3}
    For \(\bm{w} \in {\cal S}\), with notation \eqref{eq:Theta_def}, \eqref{eq:a_def} and \eqref{eq:W_def}, we have
    \begin{itemize}
        \item [(i)] \(\pmb{c}(\pmb{w})=\begin{bmatrix}
            \pmb{a}(\pmb{w})\\ \pmb{b}(\pmb{w})
        \end{bmatrix} \in \mathbb{C}^{N+d+1}\) is a solution of \eqref{eq:dual_question} if and only if it is an eigenvector of the Hermitian positive semi-definite generalized eigenvalue problem \((A_{\pmb{w}},B_{\pmb{w}})\) and \(d(\pmb{w})\) is the smallest eigenvalue satisfying
        \begin{equation}\label{eq:item3_1}
            A_{\pmb{w}}\pmb{c}(\pmb{w})=d(\pmb{w})B_{\pmb{w}}\pmb{c}(\pmb{w})\ \ and\ \ \pmb{c}(\pmb{w})^{\rm H}B_{\pmb{w}}\pmb{c}(\pmb{w})=1,
        \end{equation}
        where 
        \begin{equation}\nonumber
            A_{\pmb{w}}:=[-\pmb{\Theta},\pmb{F}\pmb{\Phi}]^{\rm H}\pmb{W}_{\otimes}[-\pmb{\Theta},\pmb{F}\pmb{\Phi}]=\begin{bmatrix}
                \pmb{\Theta}^{\rm H}\pmb{W}_{\otimes}\pmb{\Theta}&-\pmb{\Theta}^{\rm H}\pmb{W}_{\otimes}\pmb{F}\pmb{\Phi}\\-\pmb{\Phi}^{\rm H}\pmb{F}^{\rm H}\pmb{W}_{\otimes}\pmb{\Theta}&\pmb{\Phi}^{\rm H}\pmb{F}^{\rm H}\pmb{W}_{\otimes}\pmb{F}\pmb{\Phi}
            \end{bmatrix},
        \end{equation}
        \begin{equation}\nonumber
            B_{\pmb{w}}:=[0_{m \times N},\pmb{\Phi}]^{\rm H}W[0_{m \times N},\pmb{\Phi}]=\begin{bmatrix}
                0_{N \times N}&0_{N \times (d+1)}\\0_{(d+1)\times N}&\pmb{\Phi}^{\rm H}W\pmb{\Phi}
            \end{bmatrix};
        \end{equation}
        \item [(ii)] the Hermitian matrix \(H_{\pmb{w}}:=A_{\pmb{w}}-d(\pmb{w})B_{\pmb{w}}\) is positive semi-definite;
        \item [(iii)] let \(\sqrt{W}\pmb{\Phi}=Q_qR_q \in \mathbb{C}^{m \times (d+1)}\) and \(\sqrt{\pmb{W}_{\otimes}}\pmb{\Theta}=Q_pR_p \in \mathbb{C}^{ms \times N}\) be the thin QR factorizations where \(Q_q \in \mathbb{C}^{m\times \tilde{n}_2}\), \(Q_p \in \mathbb{C}^{ms \times \tilde{n}_1}\), \(R_q \in \mathbb{C}^{\tilde{n}_2 \times (d+1)}\), \(R_p \in \mathbb{C}^{\tilde{n}_1 \times N}\) with \(\tilde{n}_1={\rm rank}(\sqrt{\pmb{W}_{\otimes}}\pmb{\Theta})\) and \(\tilde{n}_2={\rm rank}(\sqrt{W}\pmb{\Phi})\). Then \(d(\pmb{w})\) is the smallest eigenvalue of the following Hermitian positive semi-definite matrix
        \begin{equation}\label{eq:eg_matrix_important}
            S(\pmb{w}):=S_{\rm F}-S_{qp}S_{qp}^{\rm H} \in \mathbb{C}^{\tilde{n}_2 \times \tilde{n}_2},
        \end{equation}
        where 
        \begin{equation}\nonumber
            S_{\rm F}=Q_q^{\rm H}\left(\sum\limits_{i=1}^{s}|F_{i}|^2\right)Q_q \in \mathbb{C}^{\tilde{n}_2 \times \tilde{n}_2},\ \ S_{qp}=Q_q^{\rm H}\pmb{F}^{\rm H}Q_p \in \mathbb{C}^{\tilde{n}_2 \times \tilde{n}_1},
        \end{equation}
        with \(|F_{i}|={\rm diag}(|t_{i}(x_1)|,\dots,|t_{i}(x_m)|) \in \mathbb{C}^{m \times m}\). Moreover, \(\sqrt{d(\pmb{w})}\) is the smallest singular value of \((I-Q_pQ_p^{\rm H})\pmb{F}Q_q \in \mathbb{C}^{ms \times \tilde{n}_2}\) with the associated singular vector \(\hat{\pmb{b}}=R_q\pmb{b}(\pmb{w})\). Also, \(R_p\pmb{a}(\pmb{w})=S_{qp}^{\rm H}\hat{\pmb{b}}\).
    \end{itemize}
\end{proposition}

\subsection{The {\sf m-d-Lawson} iteration}\label{subsec:mdlawson}
In the dual   framework, Lawson's iteration has been established as an effective method for solving the dual problem \eqref{eq:dual_question_simple} \cite{zhzz:2025}. Algorithm \ref{alg:Lawson} presents its procedure for the vector-valued rational minimax approximation\footnote{See \url{https://ww2.mathworks.cn/matlabcentral/fileexchange/181757-m_d_lawson} for the MATLAB code of  {\sf m-d-Lawson}.}, where particular attention must be paid to the numerical stability and accuracy in computing both the dual objective value $d(\pmb{w}^{(k)})$ as well as  the associated   vectors\footnote{The evaluation of $\pmb{r}_{i}^{(k)} = \pmb{p}_{i}^{(k)}./\pmb{q}^{(k)}$ demands special care when $\pmb{q}^{(k)}$ has nodal zeros. In this case,  the homogeneous pair from \eqref{eq:paris} has a common zero at a node, and the associated $r_i(x)$ should first be interpreted through an irreducible, pole-free representation on the retained node set before evaluating  $\pmb{r}_{i}^{(k)}$.} $\pmb{r}_{i}^{(k)} = \pmb{p}_{i}^{(k)}./\pmb{q}^{(k)}$ ($1 \le i \le s$) in Step 3. We use the implementation given in \cite{zhzz:2025} where the Vandermonde with Arnoldi (V+A) approach \cite{brnt:2021,hoka:2020,zhsl:2024} is employed. This results in rational components $r^*_i = p^*_i/q^*$ where the numerator and denominator polynomials are represented not in the monomial basis, but rather through a system of discrete orthogonal polynomials of increasing degree (we  {call} the V+A basis). This representation motivates our development in Section  \ref{sec:linearization} of a strong linearization framework suitable for matrix polynomials expressed in generalized graded bases. The algorithm incorporates weight updates governed by \eqref{eq:Update_rule_lawson} (Step 5), with $\beta > 0$ denoting the Lawson exponent (typically initialized at $\beta=1$). For comprehensive convergence analysis of {\sf m-d-Lawson}, we direct readers to \cite[Section 5]{zhzz:2025}.

\begin{algorithm}[h!!!]
\caption{A matrix-valued dual Lawson's iteration ({\sf m-d-Lawson}) for \eqref{eq:minimax}} \label{alg:Lawson}
\begin{algorithmic}[1]
\renewcommand{\algorithmicrequire}{\textbf{Input:}}
\renewcommand{\algorithmicensure}{\textbf{Output:}}
\REQUIRE Given \(\{(x_{\ell},\bt(x_{\ell}))\}_{\ell=1}^m\) with \(x_{\ell} \in  \bar\Omega\), a tolerance for strong duality \(\epsilon_r>0\), the maximum number \(k_{\rm maxit}\) of iterations;
\ENSURE the approximated   rational minimax approximant \({\boldsymbol{\xi}}^{(k)} \in \mathscr{R}_{s}\) of \eqref{eq:minimax} 
\smallskip
\STATE (Initialization) Let \(k=0\); choose \(\bzs<\pmb{w}^{(0)} \in {\cal S}\) and a tolerance \(\epsilon_{\pmb{w}}\) for weights;
\STATE (Filtering: optional) Remove nodes \(x_{\ell}\) with \(w^{(k)}_{\ell}<\epsilon_{\pmb{w}}\);
\STATE compute \(d(\pmb{w}^{(k)})\) and the associated \(\pmb{r}_{i}^{(k)}=\pmb{p}^{(k)}_{i}./\pmb{q}^{(k)}\) according to {Proposition \ref{prop:item3}};
\STATE (Stop rule and evaluation) Stop either if \(k \ge {k_{\rm maxit}}\) or
\begin{equation}\nonumber
    \epsilon(\pmb{w}^{(k)}):=\left|\frac{e({\boldsymbol{\xi}}^{(k)})-d(\pmb{w}^{(k)})}{e({\boldsymbol{\xi}}^{(k)})}\right|<\epsilon_r,\ {\rm where}\ \ e({\boldsymbol{\xi}}^{(k)})=\max_{x_\ell \in {\cal X}}\|\bt(x_{\ell})-{\boldsymbol{\xi}}^{(k)}(x_{\ell})\|_{2}^2;
\end{equation}
\STATE (Updating weights) Update the weight vector \(\pmb{w}^{(k+1)}\) according to 
\begin{equation}\label{eq:Update_rule_lawson}
    w^{(k+1)}_{\ell}=\frac{w^{(k)}_{\ell}{\|\bt(x_{\ell})-{\boldsymbol{\xi}}^{(k)}(x_{\ell})\|_{2}^{\beta}}}{\sum\limits_{i}w^{(k)}_i\|\bt(x_i)-{\boldsymbol{\xi}}^{(k)}(x_i)\|_{2}^{\beta}},\ \ \ \forall \ell,
\end{equation}
and go to Step 2 with \(k\leftarrow k+1\).
\end{algorithmic}
\end{algorithm}

\section{Error bound analysis}\label{sec:erroranalysis}
As mentioned earlier in Remark \ref{rmk:entry_approximate}, we do not directly approximate the larger-scale function \(T(x)\), but instead obtain first \({\boldsymbol{\xi}}^*(x)=[r^*_1(x),\dots,r^*_s(x)]^{\rm T}\) by approximately computing the best uniform approximation of its  {split coefficient vector} \(\bt(x)=[t_1(x),\dots,t_s(x)]^{\rm T}\) over the region of interest \(\bar\Omega\) via {\sf m-d-Lawson}. 
The resulting \(n \times n\)  {rational matrix function}
   $ R^*(x)=\sum\limits_{i=1}^sr^*_i(x)E_i$
is then regarded as an approximation to the original function \(T(x)\). Therefore, it is necessary to conduct a further analysis of the approximation error between \(R^*(x)\) and \(T(x)\) over \(\bar\Omega\). The following lemma is useful for the error estimate between \(R^*(x)\) and \(T(x)\).
\begin{lemma}\label{lem:lem1}
    With the constant matrices \(E_i \in \mathbb{C}^{n\times n}\), \(i=1,\dots,s\), define the following linear mapping
    \begin{equation}\label{eq:lm}
        \mathscr{A}_{\{E_i\}_i}:\bz=[z_1,\dots,z_s]^{\rm T} \in \mathbb{C}^s \mapsto \sum\limits_{i=1}^{s}z_iE_i \in \mathbb{C}^{n \times n}.
    \end{equation}
    Then for each \(\bz\in \mathbb{C}^{s}\), we have \(\left\|\mathscr{A}_{\{E_i\}_i}(\bz)\right\|_{\rm F} \le \sqrt{\left\|\mathscr{G}_{\{E_i\}_i}\right\|_2}\cdot \|\bz\|_2\), where 
    \begin{equation}\label{eq:GE_def}
        \mathscr{G}_{\{E_i\}_i}=\begin{bmatrix}
            {\rm tr}\left(E_1^{\rm H}E_1\right)&\dots&{\rm tr}\left(E_s^{\rm H}E_1\right) \\ \vdots & \ddots & \vdots \\ {\rm tr}\left(E_1^{\rm H}E_s\right)&\dots&{\rm tr}\left(E_s^{\rm H}E_s\right)
        \end{bmatrix} \in \mathbb{C}^{s \times s}.
    \end{equation}
\end{lemma}
\begin{proof}
    By definition, we have
    \begin{align}\nonumber
        0 & \le \left\|\mathscr{A}_{\{E_i\}_i}(\bz)\right\|_{\rm F}^2  ={\rm tr}\left(\left(\sum\limits_{i=1}^{s}z_iE_i\right)^{\rm H}\left(\sum\limits_{j=1}^{s}z_jE_j\right)\right) \\ \nonumber &=\sum\limits_{i=1}^{s}\sum\limits_{j=1}^{s}z_j^{\rm H}{\rm tr}\left(E_i^{\rm H}E_j\right)z_i=\bz^{\rm H}\mathscr{G}_{\{E_i\}_i}\bz \le   \|\bz\|_2^2\cdot\left\|\mathscr{G}_{\{E_i\}_i}\right\|_2,
    \end{align}
    leading to the desired result.
\end{proof}
The following theorem gives an error bound between $T$ and $R^*$. 
\begin{theorem}\label{thm:thm1}
    Suppose \({\boldsymbol{\xi}}^*(x)=[r^*_1(x),\dots,r^*_s(x)]^{\rm T}\) and \(\bt(x)=[t_1(x),\dots,t_s(x)]^{\rm T}\) satisfying \(\sup_{x \in \bar\Omega}\|\bt(x)-{\boldsymbol{\xi}}^*(x)\|_2\le \epsilon^{\Omega}_{{\boldsymbol{\xi}}^*}\), then
    \begin{equation}\label{eq:err_approx}
        \sup_{x \in\bar \Omega}\left\|T(x)-R^*(x)\right\|_{\rm F}\leq\sqrt{\left\|\mathscr{G}_{\{E_i\}_i}\right\|_2}\cdot \epsilon^{\Omega}_{{\boldsymbol{\xi}}^*},
    \end{equation}
    where \(R^*(x)=\sum\limits_{i=1}^{s}r^*_i(x)E_i\), \(T(x)=\sum\limits_{i=1}^{s}t_i(x)E_i\) and \(\mathscr{G}_{\{E_i\}_i} \in \mathbb{C}^{s \times s}\) is defined in \eqref{eq:GE_def}.
\end{theorem}
\begin{proof}
    Using Lemma \ref{lem:lem1} and with the notation in \eqref{eq:lm},   for each \(x \in \bar\Omega\),
    \begin{align}\nonumber
        \|T(x)-R^*(x)\|_{\rm F} &= \left\|\mathscr{A}_{\{E_i\}_i}(\bt(x))-\mathscr{A}_{\{E_i\}_i}({\boldsymbol{\xi}}^*(x))\right\|_{\rm F} \\ \nonumber &= \left\|\mathscr{A}_{\{E_i\}_i}(\bt(x)-{\boldsymbol{\xi}}^*(x))\right\|_{\rm F} \\ \nonumber & \le \sqrt{\left\|\mathscr{G}_{\{E_i\}_i}\right\|_2}\cdot \|\bt(x)-{\boldsymbol{\xi}}^*(x)\|_2 \le \sqrt{\left\|\mathscr{G}_{\{E_i\}_i}\right\|_2}\cdot \epsilon^{\Omega}_{{\boldsymbol{\xi}}^*}.
    \end{align}
\end{proof}
\begin{corollary}\label{coro:coro0}
    Let \({\boldsymbol{\xi}}^*\),  \(R^*\) and \(T\) be as defined in Theorem \ref{thm:thm1}, \(\mathscr{G}_{\{E_i\}_i}\) as in \eqref{eq:GE_def}, and assume \(\sup_{x \in \bar\Omega}\|\bt(x)-{\boldsymbol{\xi}}^*(x)\|_2\le \epsilon_{{\boldsymbol{\xi}}^*}^{\Omega}\). Then for any eigenpair \((\lambda, \bu)\) of \(R^*\), where \(\lambda \in \Omega\) and \(\|\bu\|_2=1\), it holds \(\|T(\lambda)\bu\|_2\le \sqrt{\|\mathscr{G}_{\{E_i\}_i}\|_2}\cdot \epsilon_{{\boldsymbol{\xi}}^*}^{\Omega}\).
\end{corollary}
\begin{proof}
The assertion can be seen from
    \begin{align}\nonumber
        \|T(\lambda)\bu\|_2&=\|(T(\lambda)-R^*(\lambda))\bu-R^*(\lambda)\bu\|_2 \\ \nonumber & \le \|(T(\lambda)-R^*(\lambda))\bu\|_2+\|R^*(\lambda)\bu\|_2 \\ \nonumber & \le \|T(\lambda)-R^*(\lambda)\|_2\cdot\|\bu\|_2\le \sqrt{\left\|\mathscr{G}_{\{E_i\}_i}\right\|_2}\cdot \epsilon_{{\boldsymbol{\xi}}^*}^{\Omega},
    \end{align}
    where the last inequality follows from Theorem \ref{thm:thm1} and the condition \(
    \|\bu\|_2=1\).
\end{proof}
\begin{remark}\label{rmk:practical_mdlawson_err}
    In practical applications of Algorithm \ref{alg:Lawson}, we can prescribe an error threshold $\epsilon^{{\cal X}}_{{\boldsymbol{\xi}}^*}>0$ to ensure $\max_{x_{\ell} \in {\cal X}}\|\bt(x_{\ell})-{\boldsymbol{\xi}}^*(x_{\ell})\|_{2}\le \epsilon^{{\cal X}}_{{\boldsymbol{\xi}}^*}$. If ${\cal X}$ is a sufficiently refined discretization of $\Gamma=\partial\Omega$, $\epsilon^{{\cal X}}_{{\boldsymbol{\xi}}^*}$ serves as a numerical approximation of $\epsilon^{\Omega}_{{\boldsymbol{\xi}^*}}$ from Theorem \ref{thm:thm1}.   Furthermore, if ${\boldsymbol{\xi}}^*(x)$ is pole-free in  $\bar\Omega$, the maximum norm principle \cite[Theorem 1]{cond:2020} guarantees that
$$\sup_{x \in {\cal X}} \|\bt(x)-{\boldsymbol{\xi}}^*(x)\|_2 \le \sup_{x \in \partial \Omega} \|\bt(x)-{\boldsymbol{\xi}}^*(x)\|_2=\sup_{x \in \bar\Omega} \|\bt(x)-{\boldsymbol{\xi}}^*(x)\|_2 \le \epsilon^{\Omega}_{{\boldsymbol{\xi}}^*}.  $$
\end{remark}

{
\begin{proposition} \label{prop:spectral_count_preservation}
Let $\Gamma=\partial\Omega$. Assume that $T$ is a regular (i.e., ${\rm det}(T(x))\not \equiv 0$) matrix-valued function analytic on  $\Omega$, and that $T(x)$ is nonsingular for all $x\in\Gamma$. Let $R^*(x)=P^*(x)/q^*(x)$ be pole-free on  $\Omega$. If
\begin{equation}\label{eq:eta_boundary_count}
    \nu_{\Gamma}:=\sup_{x\in\Gamma}
    \left\|T(x)^{-1}\bigl(T(x)-R^*(x)\bigr)\right\|_2<1,
\end{equation}
then $T$ and $R^*$ have the same number of eigenvalues in $\Omega$, counting algebraic multiplicities. Moreover, since $q^*$ has no zeros in $\Omega$, $R^*$ and $P^*$ have the same eigenvalues in $\Omega$, with the same multiplicities. 
\end{proposition}
\begin{proof} This is a consequence of \cite[Lemma 2.1]{biho:2013}. Indeed, setting $E(x)=R^*(x)-T(x)$ and $T_s(x)=T(x)+sE(x)$ for $s\in[0,1]$, we have for every $x\in\Gamma$,
\[
    T_s(x)=T(x)\left(I+sT(x)^{-1}E(x)\right),\qquad
    \left\|sT(x)^{-1}E(x)\right\|_2\leq s\cdot \nu_{\Gamma}<1.
\]
Hence $I+sT(x)^{-1}E(x)$ is nonsingular by the Neumann series criterion, and so is $T_s(x)$ on $\Gamma$ for every $s\in[0,1]$. Applying directly \cite[Lemma 2.1]{biho:2013} to the homotopy $T_s$ proves that $T=T_0$ and $R^*=T_1$ have identical eigenvalue counts in $\Omega$. The pole-free assumption implies $\det (R^*(x))=\det (P^*(x))/[q^*(x)]^n$ with a nonvanishing denominator in $\Omega$, which gives the last assertion.
\end{proof}
\begin{remark}\label{rmk:numerical_eta_estimate}
For the pole-free rational approximant $R^*$ computed from the discrete minimax problem, \eqref{eq:err_approx} implies
\begin{align}\label{eq:eta_bound_by_approx}
    \nu_{\Gamma}
    &\leq \left(\sup_{x\in\Gamma}\|T(x)^{-1}\|_2\right)
    \left(\sup_{x\in\Gamma}\|T(x)-R^*(x)\|_2\right)  \leq \left(\sup_{x\in\Gamma}\|T(x)^{-1}\|_2\right)
    \sqrt{\left\|\mathscr{G}_{\{E_i\}_i}\right\|_2}\,
    \epsilon_{{\boldsymbol{\xi}^*}}^{\Omega}.
\end{align}
In view of Remark \ref{rmk:practical_mdlawson_err}, on a sufficiently refined boundary validation set ${\cal X}_{\rm val}\subset\Gamma$ we may compute
\begin{equation}\label{eq:minimaxerr-val}
    \epsilon_{{\boldsymbol{\xi}^*}}^{{\cal X}_{\rm val}}
    :=\max_{x_{\ell}\in{\cal X}_{\rm val}}
      \|\bt(x_{\ell})-{\boldsymbol{\xi}}^*(x_{\ell})\|_2,\qquad
    \alpha_{{\cal X}_{\rm val}}
    :=\max_{x_{\ell}\in{\cal X}_{\rm val}}\|T(x_{\ell})^{-1}\|_2,
\end{equation}
and approximate   $\nu_{\Gamma} $ by 
\begin{equation}\label{eq:eta_discrete_apriori_est}
    \widehat{\nu}_{\Gamma}^{\,{\rm bd}}
    :=\alpha_{{\cal X}_{\rm val}}
      \sqrt{\left\|\mathscr{G}_{\{E_i\}_i}\right\|_2}\,
      \epsilon_{{\boldsymbol{\xi}}^*}^{{\cal X}_{\rm val}},
\end{equation}
which  uses only the discrete approximation error together with evaluations of $T$ on the validation boundary set. 
We shall report numerical experiments on this factor for verifying Proposition \ref{prop:spectral_count_preservation} in Section \ref{sec:numerical}. 
\end{remark}
}
 
\section{Linearization and the framework of {\sf NEP\_MiniMax}}\label{sec:linearization}
Let 
\begin{equation}\nonumber
    {\boldsymbol{\xi}}^*(x)=[r^*_1(x),\dots,r^*_s(x)]^{\rm T}=\left[\frac{p^*_1(x)}{q^*(x)},\dots,\frac{p^*_s(x)}{q^*(x)}\right]^{\T},
\end{equation} 
be obtained from \eqref{eq:bestfX} and  
\begin{equation}
    R^*(x)=\frac{1}{q^*(x)}\sum\limits_{i=1}^{s}p^*_i(x)E_i=:\frac{1}{q^*(x)}P^*(x),
\end{equation}
where  \(P^*(x)\) is a polynomial matrix on \(\bar\Omega\). 
Clearly, every eigenvalue of $R^*(x)$ must also be an eigenvalue of $P^*(x)$, though the converse does not always hold. Importantly, {\sf m-d-Lawson} delivers an approximately best uniform approximation for $\bt(x)=[t_1(x),\dots,t_s(x)]^{\rm T}$ over $\bar\Omega$. As demonstrated numerically in Section \ref{sec:numerical}, the approximant ${\boldsymbol{\xi}}^*(x)$ computed via {\sf m-d-Lawson} is typically pole-free in $\Omega$ and on its boundary, i.e., ${\rm zer}(q^*) \cap \bar\Omega=\emptyset$.  {For the linearization analysis and the conversion from the REP to the PEP, we explicitly assume hereafter that $q^*$ has no zeros on $\bar\Omega$ and that $P^*$ is regular, i.e., $\det (P^*(x))\not\equiv 0$.} Consequently, 
$\Omega \cap \Lambda(R^*) = \Omega \cap \Lambda(P^*),$
indicating that the favorable approximation properties of {\sf m-d-Lawson} allow the rational eigenvalue problem (REP) for $R^*$ to be converted into a polynomial eigenvalue problem (PEP) for $P^*$. Thus, approximating the eigenpairs of $T$ on $\Omega$ reduces to solving 
$$P^*(\lambda)\pmb{u} = 0, \quad (\lambda,\pmb{u}) \in \Omega \times \left(\mathbb{C}^n \setminus \left\{\bzs\right\}\right).  $$

To numerically solve the aforementioned PEP, stability considerations are essential. For instance, an implementation of {\sf m-d-Lawson} proposed in \cite{zhzz:2025} utilizes the Vandermonde with Arnoldi (V+A) method \cite{brnt:2021}. This approach expresses the polynomials $p_i^*(x)$ and $q^*(x)$ in terms of a discrete orthogonal polynomial basis \cite{brnt:2021,zhsl:2024,zhna:2026} $\{\vartheta_j(x)\}_{j=0}^\gamma$ constructed with respect to the sample nodes $\mathcal{X}$.  {Section \ref{subsec:linearP} constructs the particular V+A recurrence-based pencil needed here, transforming the problem to a  GEP  $(C_0, C_1)$. Furthermore, under the regularity assumption on $P^*$ stated above, Section \ref{subsec:stronglinear} verifies directly that this pencil is strong.}

\subsection{V+A basis and the representation of \(P^*(x)\)}\label{subsec:RepP}
Let $$\gamma=\max\left\{\max_{1 \le i \le s}\{n_{i}\},d\right\},$$ where \(n_{i}\ \forall 1\le i \le s\) and \(d\) are the degree specified by the user for approximation of \(p^*_{i}(x)\) and \(q^*(x)\), respectively. According to the Vandermonde with Arnoldi (V+A) \cite{brnt:2021,zhsl:2024,zhna:2026}, we have a discrete orthogonal polynomial (V+A) basis $\{\vartheta_j(x)\}_{j=0}^\gamma$ satisfying \(\vartheta_j(x) \in \mathbb{P}_{j}\ \forall 1\le  j\le  \gamma\)  and 
\begin{equation}\label{eq:recurrence_relation}
    x\left[\vartheta_{0}(x) ,\vartheta_{1}(x) ,\dots,\vartheta_{\gamma}(x) \right] =\left[\vartheta_{0}(x) ,\vartheta_{1}(x) ,\dots,\vartheta_{\gamma}(x) ,\vartheta_{\gamma+1}(x) \right] H,
\end{equation}
where 
\begin{equation}\label{eq:H_def}
    H=\begin{bmatrix}h_{1,1}&h_{1,2}&\dots&\dots&h_{1,\gamma}&h_{1,\gamma+1}\\ h_{2,1}&h_{2,2}&\dots&\dots&h_{2,\gamma}&h_{2,\gamma+1}\\ &h_{3,2}&\dots&\dots&h_{3,\gamma}&h_{3,\gamma+1}\\ &&\ddots&&\vdots&\vdots\\ &&&\ddots&h_{\gamma,\gamma}&h_{\gamma,\gamma+1}\\ &&&&h_{\gamma+1,\gamma}&h_{\gamma+1,\gamma+1}\\&&&&&h_{\gamma+2,\gamma+1}\end{bmatrix} \in \mathbb{C}^{(\gamma+2)\times(\gamma+1)}
\end{equation}
is an irreducible (i.e., $h_{i+1,i}\ne 0$) upper-Hessenberg matrix obtained from the V+A process \cite{brnt:2021,zhsl:2024,zhna:2026}. In particular, if the nodes $\mathcal{X} = \{x_j\}_{j=1}^m \subset \mathbb{R}$ are real-valued, the matrix $H$ reduces to a tridiagonal form; moreover, it can  {be} proved using a complex version of the implicit Q theorem \cite[Theorem 7.4.2]{govl:2013} that this property extends to the complex case when all nodes $\mathcal{X} \subset \mathbb{C}$ lie on a straight line in the complex plane.
With this V+A basis,   we can write \(p^*_i(x)=\sum\limits_{j=0}^{\gamma}a^*_{i,j}\vartheta_j(x)\), and consequently
\begin{align}\label{eq:P_T_def} 
    P^*(x)=\sum\limits_{i=1}^{s}\left(\sum\limits_{j=0}^{\gamma}a^*_{i,j}\vartheta_j(x)\right)E_i   = \sum\limits_{j=0}^{\gamma}\vartheta_j(x)\left(\sum\limits_{i=1}^{s}a^*_{i,j}E_i\right)=:\sum\limits_{j=0}^{\gamma}\vartheta_j(x)A_j.
\end{align}
{Throughout the remainder of this section, $\gamma$ is the \emph{grade} assigned to the regular matrix polynomial $P^*$, rather than an assertion that its actual degree is exactly $\gamma$. Since the V+A basis is degree-graded, if $k_\gamma\ne 0$ denotes the leading scalar coefficient of $\vartheta_\gamma$, then the coefficient of $x^\gamma$ in the monomial representation of $P^*$ is $k_\gamma A_\gamma$. Hence $\deg P^*\leq \gamma$. Note that the matrix $A_\gamma$ may be singular, and cancellation in $\sum_{i=1}^s a^*_{i,\gamma}E_i$ may even yield $A_\gamma=0$. Such cases correspond to possible eigenstructure at infinity relative to grade $\gamma$ and are included below through the reversal condition in the definition of strong linearization.}
\begin{remark}
It is worth noting that as long as $H$ is irreducible,  $\{\vartheta_i(x)\}_{i=0}^{\gamma}$ forms a family of polynomials with strictly increasing degrees. That is, $\{\vartheta_i(x)\}_{i=0}^{\gamma}$ is a degree-graded system.  Notably, common bases such as monomial bases and orthogonal polynomial bases (including Chebyshev polynomials) satisfy these recurrence relations as special cases. 
\end{remark}

\subsection{A linearization of \(P^*(x)\)}\label{subsec:linearP}
For the linearization of \(P^*(x)=\sum\limits_{j=0}^{\gamma}\vartheta_j(x)A_j\) in \eqref{eq:P_T_def},  we denote by \(k_j\) the coefficient of the highest degree  in \(\vartheta_j(x)\) with   \(k_0=1\), i.e. \(\vartheta_0(x) \equiv 1\). Then by the recurrence relation \eqref{eq:recurrence_relation},  it holds that 
\begin{equation}\label{eq:coeffs_recurrence}
    h_{j+1,j}=\frac{k_{j-1}}{k_{j}},\ \ j=1,2,\dots,\gamma.
\end{equation}
The following lemma plays an important role in constructing the linearization of \(P^*\).
\begin{lemma}\label{lem:linearization}
    Define the matrix pencil $(C_0, C_1)$ as follows:
$$C_0 := \begin{bmatrix}
H([\gamma],[\gamma-1])^{\rm T} \otimes I_n \\ 
-k_{\gamma-1}[A_0,A_1,\dots,A_{\gamma-1}] + k_{\gamma}H([\gamma],\gamma)^{\rm T}\otimes A_{\gamma}
\end{bmatrix},
\quad
C_1 := \begin{bmatrix}
I_{(\gamma-1)n} & \\
 & k_{\gamma}A_{\gamma}
\end{bmatrix},$$
where $H([\gamma],[\gamma-1])$ denotes the submatrix consisting of the first $\gamma$ rows and first $\gamma-1$ columns of $H$, while $H([\gamma],\gamma)$ represents the $\gamma$-th column of $H$ as defined in \eqref{eq:H_def}. Then 
\begin{equation}\nonumber
    \left(C_0-xC_1\right)\begin{bmatrix}
        \vartheta_0(x)I_n \\ \vartheta_1(x)I_n \\ \vdots \\ \vartheta_{\gamma-1}(x)I_n
    \end{bmatrix}=\begin{bmatrix}
        \bzs \\ \bzs \\ \vdots \\ -k_{\gamma-1}P^*(x)
    \end{bmatrix},
\end{equation}
or equivalently,
\begin{equation}\label{eq:lin_important}
    \left(C_0-xC_1\right)\left(\pmb{\vartheta}(x)\otimes I_n\right)=-k_{\gamma-1}\left(\pmb{e}_{\gamma}\otimes P^*(x)\right),
\end{equation}
where \(\pmb{\vartheta}(x)=[\vartheta_0(x),\dots,\vartheta_{\gamma-1}(x)]^{\rm T}\).
\end{lemma}
\begin{proof}
    By  calculations, we obtain 
    \begin{equation}\label{eq:sss}
        \left(C_0-xC_1\right)\left(\pmb{\vartheta}\otimes I_n\right)=\begin{bmatrix}
            (x \vartheta_{0}-\sum\limits_{j=1}^{2}h_{j,1}\vartheta_{j-1})I_{n}\\ \vdots \\ (x \vartheta_{\gamma-2}-\sum\limits_{j=1}^{\gamma}h_{j,\gamma-1}\vartheta_{j-1})I_{n} \\ -k_{\gamma-1}\sum\limits_{j=0}^{\gamma-1}\vartheta_{j}A_{j}-k_{\gamma}\Big(x \vartheta_{\gamma-1}-\sum\limits_{j=1}^{\gamma}h_{j,\gamma}\vartheta_{j-1}\Big)A_{\gamma}
        \end{bmatrix}.
    \end{equation}
    With the recurrence relation \eqref{eq:recurrence_relation}, the first \(\gamma-1\) blocks in \eqref{eq:sss} are all 0. Finally, by \eqref{eq:recurrence_relation}   and \(h_{\gamma+1,\gamma}=\frac{k_{\gamma-1}}{k_{\gamma}}\), we get  $$-k_{\gamma-1}\sum\limits_{j=0}^{\gamma-1}\vartheta_{j}(x)A_{j}-k_{\gamma}\Big(x \vartheta_{\gamma-1}(x)-\sum\limits_{j=1}^{\gamma}h_{j,\gamma}\vartheta_{j-1}(x)\Big)A_{\gamma}=-k_{\gamma-1}P^*(x)$$
    as desired.
\end{proof}
The next result then reveals the relationship between the eigenvalues of \(P^*\) and linear matrix pencil \((C_0,C_1)\):
\begin{corollary}\label{prop:linearization}
    Every finite eigenvalue of \(P^*\) is also an eigenvalue of the matrix pencil \((C_0,C_1)\) defined in \eqref{eq:lin_important}.
\end{corollary}
\begin{proof}
    Assume \((\lambda,\bu)\) is an eigenpair of \(P^*\), i.e. \(P^*(\lambda)\bu=0\). Using Lemma \ref{lem:linearization},   we  {obtain}
    \begin{equation}\label{eq:eigenvector_huifu}
        \left(C_0-\lambda C_1\right)\left(\pmb{\vartheta}(\lambda)\otimes \bu\right)=-k_{\gamma-1}\pmb{e}_{\gamma}\otimes \left(P^*(\lambda)\bu\right)=\bzs,
    \end{equation}
    leading to the desired result.
\end{proof}
\begin{remark}\label{rmk:nep_pencil_relation}
    Through the proof of Corollary \ref{prop:linearization}, we find that if \((\lambda,\bu)\) is the eigenpair of \(P^*\), then \((\lambda,\pmb{\vartheta}(\lambda)\otimes \bu)\) is the eigenpair of the matrix pencil \((C_0,C_1)\), where \(\pmb{\vartheta}(x)\) is defined in \eqref{eq:lin_important}.  {Conversely, let $(\lambda,\bv)$ be a finite eigenpair of
$(C_0,C_1)$, and partition
$\bv=[\bv_1^{\rm T},\ldots,\bv_\gamma^{\rm T}]^{\rm T}$.
The first $\gamma-1$ block equations, together with
$h_{j+1,j}\ne0$ and the recurrence relation
\eqref{eq:recurrence_relation}, imply
$
\bv_i=\vartheta_{i-1}(\lambda)\bv_1,
~ i=1,\ldots,\gamma.
$
In particular, $\bv_1\ne0$. Hence
$\bv=\pmb{\vartheta}(\lambda)\otimes\bv_1$, and
\eqref{eq:lin_important} yields
$P^*(\lambda)\bv_1=0$.
Therefore, $\lambda$ is an eigenvalue of $P^*$ and $\bv_1$
is an associated eigenvector.} Moreover,  \(\forall i=2,\dots,\gamma\), \(\pmb{v}_i\) can be obtained from
    \begin{equation}\label{eq:eigenvector_recover}
        \pmb{v}_i=\vartheta_{i-1}(\lambda)\pmb{v}_1,\ \ i=2,3,\dots,\gamma.
    \end{equation}
\end{remark}

\subsection{The linearization is strong}\label{subsec:stronglinear}
While Corollary \ref{prop:linearization} provides a method for computing eigenpairs of $P^*$, a critical question remains: whether $C_0-\lambda C_1$ defined in \eqref{eq:lin_important} constitutes a strong linearization \cite[Definition 2.1]{amcl:2009} of $P^*(\lambda)$. This property is particularly important because strong linearizations preserve the complete spectral information of the original matrix polynomial, including eigenvalue multiplicities \cite{lanc:2008}. Such preservation is essential for rigorous analysis of eigenvalue distributions, stability properties, and related problems. 
{In principle, the strong linearization property of $ P^*(\lambda) $ can be derived from the general treatment \cite{fasa:2017,nnt:2017} of the generalized ansatz space with an `ansatz vector' associated with \eqref{eq:lin_important}. However, in our specific case, since the particular linearization structure of $ P^*(\lambda) $ will be used to solve the GEP  associated with the pencil $(C_0, C_1) $, and since care must be taken regarding the singular or zero cases of the  monomial leading coefficient $P_{\gamma} $, we provide a self-contained verification of our linearization.}

For any polynomial $P(\lambda)$  having grade $\gamma$, we define its reverse polynomial as
$P^{\#}(\lambda) := \lambda^{\gamma} P\left( \lambda^{-1} \right).$
The concept of strong linearization is formally characterized as follows:
\begin{definition}{\rm ({grade-$\gamma$ formulation of} \cite[Definition 2.1]{amcl:2009})}
    A \(\gamma n \times \gamma n\) linear matrix pencil \(\lambda A-B\)\ is a strong linearization of \(n \times n\) regular matrix polynomial \(P(\lambda)\) of  {grade} \(\gamma\) if there are unimodular matrix polynomials\footnote{$D(x):\Omega\rightarrow \bbC^{n\times n}$ is a unimodular matrix-valued function  if   $\det (D(x))$ is a nonzero
constant over $\Omega$.} \(D(\lambda)\) and \(G(\lambda)\)\ such that 
\begin{equation}\label{eq:global analytic of p}
    \begin{bmatrix}P(\lambda)&\\ &I_{(\gamma-1)n}\end{bmatrix}=D(\lambda)(\lambda A-B)G(\lambda),
\end{equation}
and there are unimodular matrix polynomials \(H(\lambda)\) and \(K(\lambda)\) such that
\begin{equation}
    \begin{bmatrix}P^{\#}(\lambda)&\\ &I_{(\gamma-1)n}\end{bmatrix}=H(\lambda)(A-\lambda B)K(\lambda).
\end{equation}
\end{definition} 
If condition \eqref{eq:global analytic of p}  holds without any reference to spectral behavior at infinity (if any), then, the matrix pencil $\lambda A - B$ is referred to as a weak linearization \cite{amcl:2009} of the associated polynomial.

The following LU factorization is useful for solving the GEP $(C_0,C_1)$ and is also essential to prove that \(\lambda C_1-C_0\) defined in \eqref{eq:lin_important} is a strong linearization of \(P^*(\lambda)\).
\begin{lemma}\label{lem:lu}
    Let the pencil \((C_0,C_1)\) be given in \eqref{eq:lin_important}.
    \begin{itemize}
        \item [(i)] Define permutation matrix \(S\) as 
        \begin{equation}\label{eq:S_permutation}
            S:=\begin{bmatrix}
                &&&& I_n \\ I_n &&&& \\ & I_n &&& \\ && \ddots && \\ &&& I_n &
            \end{bmatrix}.
        \end{equation}
        Then we have
        \begin{equation}\label{eq:lu1}
            \left(\lambda C_1-C_0\right)S=L(\lambda)U(\lambda),
        \end{equation}
        where
        \begin{equation}\label{eq:lu1_U_def}
            U(\lambda)=\begin{bmatrix}
                I_n&&&-\frac{\vartheta_1(\lambda)}{\vartheta_0(\lambda)}I_n \\ & I_n && \vdots \\ && \ddots & -\frac{\vartheta_{\gamma-1}(\lambda)}{\vartheta_0(\lambda)}I_n \\ &&& I_n
            \end{bmatrix},
        \end{equation}
        \begin{equation}\label{eq:lu1_L_def}
            L(\lambda)=\begin{bmatrix}
                -h_{2,1}I_n&&&&&& \\ (\lambda-h_{2,2})I_n & -h_{3,2}I_n&&&&& \\ \vdots & \vdots & \ddots &&&& \\ -h_{2,\gamma-2}I_n & -h_{3,\gamma-2}I_n & \dots & -h_{\gamma-1,\gamma-2}I_n && \\ -h_{2,\gamma-1}I_n & -h_{3,\gamma-1}I_n & \dots &  (\lambda-h_{\gamma-1,\gamma-1})I_n & -h_{\gamma,\gamma-1}I_n&\\
                l_{1,\gamma}&l_{2,\gamma}&\dots&l_{\gamma-2,\gamma}&l_{\gamma-1,\gamma}&l_{\gamma,\gamma}
            \end{bmatrix},
        \end{equation}
        and
        \begin{equation}    \label{eq:lll_def}
            l_{i,\gamma}=\begin{cases}k_{\gamma -1}A_{i}-k_{\gamma}h_{i+1,\gamma}A_{\gamma},&i=1\colon (\gamma -2),\\ k_{\gamma -1}A_{\gamma -1}-k_{\gamma}h_{\gamma ,\gamma}A_{\gamma}+\lambda k_{\gamma}A_{\gamma},&i=\gamma -1,\\ \frac{1}{\prod_{i=1}^{\gamma -1} h_{i+1,i}} P^*\left( \lambda \right) ,&i=\gamma.\end{cases}
        \end{equation}
        \item [(ii)] \(\lambda C_1-C_0=L^{\#}(\lambda)U^{\#}(\lambda)\), where
        \begin{equation}\label{eq:lu2_U_def}
            U^{\#}(\lambda)=\begin{bmatrix}
                I_n& -\frac{\vartheta_0(\lambda)}{\vartheta_1(\lambda)}I_n&&\\& I_n & \ddots &&\\ &&& \\ && \ddots & -\frac{\vartheta_{\gamma-2}(\lambda)}{\vartheta_{\gamma-1}(\lambda)}I_n \\ &&& I_n
            \end{bmatrix},
        \end{equation}
        \begin{equation}\label{eq:lu2_L_def}
            L^{\#}(\lambda)=\begin{bmatrix}
                h_{2,1}\frac{\vartheta_1}{\vartheta_0}I_n &&&&&& \\ u_{1,2} & h_{3,2}\frac{\vartheta_2}{\vartheta_1}I_n&&&&& \\ u_{1,3} & u_{2,3} & h_{4,3}\frac{\vartheta_3}{\vartheta_2}I_n &&&& \\ \vdots & \vdots & \vdots & \ddots &&& \\ u_{1,\gamma-2}&u_{2,\gamma-2}&u_{3,\gamma-2}&\dots&h_{\gamma-1,\gamma-2}\frac{\vartheta_{\gamma-2}}{\vartheta_{\gamma-3}}I_n&& \\ u_{1,\gamma-1}&u_{2,\gamma-1}&u_{3,\gamma-1}&\dots&u_{\gamma-2,\gamma-1}&h_{\gamma,\gamma-1}\frac{\vartheta_{\gamma-1}}{\vartheta_{\gamma-2}}I_n&\\u_{1,\gamma}&u_{2,\gamma}&u_{3,\gamma}&\dots&u_{\gamma-2,\gamma}&u_{\gamma-1,\gamma}&u_{\gamma,\gamma}
            \end{bmatrix}
        \end{equation}
        with
        \begin{equation}\label{eq:uuu_def}
            \begin{cases}u_{1,j}=-h_{1,j}I_{n},&j=2,\dots,\gamma-1,\\ u_{i,\gamma}=k_{\gamma -1}A_{0}-k_{\gamma}h_{1,\gamma}A_{\gamma},&i=1,\\ u_{i+1,\gamma}=u_{i,\gamma}\frac{\vartheta_{i-1}}{\vartheta_{i}} +k_{\gamma -1}A_{i}-k_{\gamma}h_{i+1,\gamma}A_{\gamma},&i=1,\dots,\gamma -2,\\ u_{i,\gamma}=\frac{\vartheta_{0}}{\vartheta_{\gamma -1} \prod_{k=1}^{\gamma -1} h_{k+1,k}} P^*\left( \lambda \right) ,&i=\gamma ,\\ u_{i,j}=(-h_{i,j}-\sum_{k=0}^{i-2} h_{k+1,j}\frac{\vartheta_{k}}{\vartheta_{i-1}} )I_{n},&i=2,\dots,\gamma -2,\  j=i+1,\dots,\gamma -1.\end{cases}
        \end{equation}
    \end{itemize}
\end{lemma}
\begin{proof}
    With the recurrence relation \eqref{eq:recurrence_relation}, \(h_{j+1,j}=\frac{k_{j-1}}{k_{j}} \ \forall j=1,2,\dots,\gamma\) and \(k_0=1\), all the results can be obtained by   calculations. Details are omitted here.
\end{proof}
With the aid of  Lemma \ref{lem:lu}, we know that  \(\lambda C_{1}-C_{0}\) is a strong linearization of \(P^*(\lambda)\). 
\begin{proposition}\label{prop:stronglinear}
{Assume that $P^*(\lambda)$ in \eqref{eq:P_T_def} is regular and is regarded as a matrix polynomial of grade $\gamma$.} The pencil \((C_0,C_1)\) in \eqref{eq:lin_important} is a strong linearization of \(P^*(\lambda)\) in \eqref{eq:P_T_def}.
\end{proposition}
\begin{proof}
See Appendix. 
\end{proof}

\subsection{The framework of {\sf NEP\_MiniMax}  for NEP \eqref{eq:nep}}\label{subsec:algframework}
Using {\sf m-d-Lawson} and linearization techniques, we now describe the framework of our proposed {\sf NEP\_MiniMax} (Algorithm \ref{alg:linear_mdlawson}) for solving NEP \eqref{eq:nep}. Letting $\bn=[n_1,\dots,n_s]^{\T}$, we denote by $${\boldsymbol{\xi}}^*=\text{\sf m-d-Lawson}(\bt,{\cal X},\bn,d)$$ the process of calling $\text{\sf m-d-Lawson}$ to have   {an approximately best} rational approximation  of $\bt(x)$ over the sample set ${\cal X} = \{x_{\ell}\}_{\ell=1}^m$, where each pair $(p_i^*(x), q^*(x))$ has type $(n_i, d)$. For simplicity, we set $n_i = d$ for rational approximation and $n_i = k$ (with $d = 0$) for polynomial approximation. During the approximation stage (lines 2--5 of Algorithm \ref{alg:linear_mdlawson}), the degree is iteratively increased until $\sqrt{e({\boldsymbol{\xi}}^*)}$ satisfies $\sqrt{e({\boldsymbol{\xi}}^*)} \le \epsilon_{{\boldsymbol{\xi}^*}}^{\cal X}$. For eigenvalue computation, small-scale matrix pencils $(C_0, C_1)$ are resolved using MATLAB's {\sf eig} function, whereas large-scale problems necessitate specialized methods \cite{chli:2025,saem:2020,tasa:2024} to mitigate computational demands. The {Kronecker-product} structure of $(C_0, C_1)$ \eqref{eq:lin_important} in our framework facilitates efficient eigenvalue computation, as elaborated in next section.

\begin{algorithm}[h!!!]
\caption{The framework of {\sf NEP\_MiniMax} for \eqref{eq:nep}} \label{alg:linear_mdlawson}
\begin{algorithmic}[1]
\renewcommand{\algorithmicrequire}{\textbf{Input:}}
\renewcommand{\algorithmicensure}{\textbf{Output:}}
\REQUIRE Given a matrix-valued function \(T(x)={\sum}_{i=1}^st_i(x)E_i\) with sample \({\cal X}=\{x_{\ell}\}_{\ell=1}^m \subset   \bar\Omega\), a tolerance  \(\epsilon_{{\boldsymbol{\xi}}^*}^{\cal X}>0\), the maximum degree \(d_{\max}\) for approximation.
\ENSURE The computed eigenpairs \((\lambda^*,\pmb{u}^*)\) of \(T(x)\).
\smallskip
\STATE Let \(\bt(x)=[t_1(x),\dots,t_s(x)]^{\rm T}\);
\STATE \textbf{for} $k$ = 1 {\sf to} \(d_{\max}\) \textbf{do}
\STATE \ \ \ \ \ \({\boldsymbol{\xi}}^*=\) {\sf m-d-Lawson}\((\bt,{\cal X}, k \be, k)\);
\STATE \ \ \ \ \ Stop if \(\sqrt{e({\boldsymbol{\xi}}^*)}\le \epsilon_{{\boldsymbol{\xi}}^*}^{\cal X}\);
\STATE \textbf{end for}
\STATE Let \(R^*(x)=\frac{1}{q^*(x)}{\sum}_{i=1}^{s}p^*_i(x)E_i\) and construct pencil \((C_0,C_1)\) according to \eqref{eq:lin_important};
\STATE Compute \((C_0,C_1)\)'s eigenpairs \((\lambda^*,\bv^*)\), then \(\pmb{u}^*=\bv^*(1:n,:)\).
\end{algorithmic}
\end{algorithm}

\section{Solving the resulting large-scale matrix pencils}\label{sec:largescale}
As previously discussed, the eigenvalue problem for an $n \times n$  {regular} matrix polynomial $P^*(\lambda)$ of  {grade} $\gamma$ can be reformulated as a  GEP  $(C_0,C_1)${---}defined in \eqref{eq:lin_important}{---}with dimension $\gamma n \times \gamma n$. For small-to-medium-scale problems (where $\gamma n$ is computationally tractable),  one can directly apply the QZ algorithm to $(C_0,C_1)$ to compute all eigenvalues. For large-scale problems, however, memory-efficient subspace iteration methods \cite{saem:2020,tasa:2024} or rational filtering techniques \cite{chli:2025}{---}the latter defined on a target region $\Omega${---}become necessary alternatives. Additionally, when the matrix pencil exhibits structured forms, more specialized approaches are available. In particular,  \cite{bemm:2015} demonstrates that for pencils with Kronecker product structure, the compact rational Krylov method (CORK) provides an effective framework for eigenvalue computation.


\subsection{Rational filter for matrix pencil}\label{subsec:rationafilter}
Recall that $\Omega \subset \mathbb{C}$ is a Jordan domain with boundary $\Gamma = \partial\Omega$ consisting of a simple Jordan curve. The indicator function of $\Omega$ can be expressed through contour integration \cite{gupt:2015,yicy:2017} as
\begin{equation}\label{eq:filter_def1}
    \rho(x) = \frac{1}{2\pi\mathrm{i}}\oint_{\Gamma}\frac{1}{z-x}\mathrm{d}z = 
    \begin{cases}
        1, & x \in \Omega, \\ 
        0, & x \not \in {\bar\Omega},
    \end{cases}
\end{equation}
which naturally defines a rational filter for $\Omega$.  The filter $\rho(x)$ enables selective extraction of eigenvalues within $\Omega$. Indeed,  for any diagonalizable matrix pencil $(A,B)$ satisfying $AX = BX\Lambda$, where $\Lambda$ contains the eigenvalues on its diagonal, the filtered matrix becomes \cite{yicy:2017}
\[
\rho(B^{-1}A) = X \mathbf{1}_{\Omega}(\Lambda) X^{-1},
\]
with $\mathbf{1}_{\Omega}(\cdot)$ denoting the eigenvalue indicator function. Crucially, $\rho(B^{-1}A)$ projects onto $\operatorname{span}\{\boldsymbol{u}_1,\dots,\boldsymbol{u}_k\}$, where $\{\boldsymbol{u}_i\}_{i=1}^k$ are eigenvectors corresponding to eigenvalues of $(A,B)$ lying within $\Omega$. Remarkably, this contour integral approach remains valid even for defective non-Hermitian pencils \cite{yicy:2017} and yields the method HFEAST \cite{ying:2019}. 
These results underlie the {\sf SIF} (Subspace Iteration with Filter) method (Algorithm~\ref{alg:si_filter}) for computing eigenvalues of large-scale pencils $(C_0,C_1)$. 

\begin{algorithm}[h!!!]
\caption{Subspace iteration with filter \cite{chli:2025,ying:2019}} \label{alg:si_filter}
\begin{algorithmic}[1]
\renewcommand{\algorithmicrequire}{\textbf{Input:}}
\renewcommand{\algorithmicensure}{\textbf{Output:}}
\REQUIRE A matrix pencil \((C_0,C_1)\),  \(\Omega\), number of columns \(\tilde{n}\), shift \(\sigma\).
\ENSURE All approximate eigenpairs \((\tilde{\lambda}_i,\tilde{\pmb{y}}_i)\), \(\tilde{\lambda}_i \in \Omega\).
\smallskip
\STATE Generate random \(Y \in \mathbb{C}^{n\gamma \times \tilde{n}}\);
\STATE \textbf{while}  {not converged} \textbf{do}
\STATE \ \ \ \ \ \(U=\rho(C_1^{-1}C_0)Y\);
\STATE \ \ \ \ \ \(V={\rm orth}(U)\);
\STATE \ \ \ \ \ \(W={\rm orth}((C_0-\sigma C_1)V)\);
\STATE \ \ \ \ \ \(\{(\tilde{\lambda}_i,\tilde{\pmb{x}}_i)\}_{i=1}^{\tilde{n}}={\rm eig}(W^{\rm H}C_0V,W^{\rm H}C_1V)\);
\STATE \ \ \ \ \ \(\tilde{\pmb{y}}_i=V\tilde{\pmb{x}}_i\), and \(Y=[\tilde{\pmb{y}}_1,\dots,\tilde{\pmb{y}}_{\tilde{n}}]\);
\STATE \textbf{end while}
\end{algorithmic}
\end{algorithm}

\begin{remark}  {The following remarks explain} Algorithm \ref{alg:si_filter}:
    \item [(1)] In Line 4,  {$V = \mathrm{orth}(U)$} denotes an orthonormal basis for the range of  {$U$}, typically computed via thin QR decomposition of  {$U$}; the same applies for Line 5. The left eigenspace basis $W$ is constructed as $W = \mathrm{orth}(C_0V - \sigma C_1V)$, where $\sigma$ is a shift parameter chosen to avoid the eigenvalues of $(C_0,C_1)$. For implementation details, we refer to the HFEAST algorithm \cite{ying:2019}.
    \item [(2)] The eigenpairs \((\tilde{\lambda}_i, \tilde{\bm{x}}_i)\) of the reduced \(\tilde{n}\times \tilde{n}\) generalized eigenvalue problem (Line~6) can be computed using MATLAB's \textsf{eig}, as \(\tilde{n}\) is typically small enough for dense direct methods.
    \item [(3)] Under exact arithmetic conditions, a single iteration of the rational filter $\rho(C_1^{-1}C_0)$ applied to $Y$ would yield a subspace $V$ that exactly spans the desired invariant subspace. In such idealized scenarios, the computed Ritz pairs $(\tilde{\lambda}_i, \tilde{\bm{y}}_i)$ from the reduced problem precisely coincide with eigenpairs of the original pencil $(C_0,C_1)$.
    \end{remark}

We emphasize that the filter $\rho(x)$ defined in \eqref{eq:filter_def1} involves a contour integral and thus requires numerical approximation in practice. For the special case when $\Gamma$ is a circular contour with center $c$ and radius $r$, we parameterize it using $c+re^{i\theta}$ ($\theta\in[0,2\pi)$), reducing \eqref{eq:filter_def1} to a one-dimensional integral amenable to standard quadrature methods. Following the approach in \cite{chli:2025}, we employ the $k$-point trapezoidal rule, which yields the discrete approximation
\begin{equation}\label{eq:kth_trap_qua}
    \zeta_{c,r,k}(x)=\sum\limits_{j=1}^{k}\frac{g^{(k)}_j}{s^{(k)}_j-x},
\end{equation}
where the poles $s_j^{(k)}$ and weights $g_j^{(k)}$ are given by
\begin{equation}\label{eq:pole_weight}
 s^{(k)}_j=c+r{ e}^{{\rm i}\theta^{(k)}_j},\ \ g^{(k)}_j=\frac{r}{k}{ e}^{{\rm i}\theta^{(k)}_j},\ \ \theta^{(k)}_j=\frac{(2j-1) \pi}{k}.
 \end{equation}
Consequently, the action of the filter on $Y$ is approximated as
\begin{equation}\label{eq:filter_matrix_crk}
\rho(C_1^{-1}C_0)Y \approx \zeta_{c,r,k}(C_1^{-1}C_0)Y = \sum_{j=1}^k g_j^{(k)} \left(s_j^{(k)} C_1 - C_0 \right)^{-1} C_1 Y.
\end{equation}
For rigorous error analysis, convergence rates, and generalizations to composite quadrature rules, we refer readers to \cite{chli:2025}.

\subsection{Exploiting block structure in the shift-invert systems}\label{subsec:shiftinv}
The solution of \eqref{eq:filter_matrix_crk} involves solving $k$ linear systems of dimension $\gamma n \times \gamma n$, which inherently depends on shift-invert iteration. Specifically, for a given shift $\mu$ and initial matrix $Y \in \mathbb{C}^{{\gamma n}\times \tilde{n}}$, each shift-invert step computes $Z$ via the linear system:
\begin{equation}\label{eq:sii_def}
(\mu C_1 - C_0)Z = C_1Y.
\end{equation}
While a standard implementation requires solving a full $\gamma n \times \gamma n$ system per iteration, our framework leverages the structure of the pencil $(C_0,C_1)${---}particularly the block LU decomposition in \eqref{eq:lu1}{---}to reduce the computational cost to solving \eqref{eq:sii_def}.

By   the LU factorization from Lemma \ref{lem:lu}(i), \eqref{eq:sii_def} becomes equivalent to
\begin{equation}\nonumber
\left(L(\mu)U(\mu)\right)S^{\rm T}Z = C_1 Y,
\end{equation}
where $L(\mu)$ and $U(\mu)$ are derived by substituting $\lambda = \mu$ into $L(\lambda)$ in \eqref{eq:lu1_L_def}  and $U(\lambda)$ in \eqref{eq:lu1_U_def}, respectively. Defining $X := U(\mu)S^{\rm T}Z$, we solve for $Z$ via
\begin{equation}\nonumber
S^{\rm T}Z = U^{-1}(\mu)X, \
L(\mu)X = C_1 Y.
\end{equation}
Partition $X = [X_{1}^{\rm T}, \dots, X_{{\gamma}}^{\rm T}]^{\rm T}$ and $Y = [Y_{1}^{\rm T}, \dots, Y_{{\gamma}}^{\rm T}]^{\rm T}$, with $X_{i}, Y_{i} \in \mathbb{C}^{n \times \tilde{n}}$. Using the structures of $C_1$, $U(\mu)$, and $L(\mu)$, \eqref{eq:lin_important}, \eqref{eq:lu1_U_def}, \eqref{eq:lu1_L_def} and noting $\vartheta_0 \equiv 1$, direct computation yields 
\begin{equation}\label{eq:wk1_iterate}
    S^{\rm T}Z=\begin{bmatrix}
        X_{{1}}+\vartheta_1(\mu)X_{{\gamma}} \\ \vdots \\ X_{{\gamma-1}}+\vartheta_{\gamma-1}(\mu)X_{{\gamma}} \\ X_{{\gamma}} 
    \end{bmatrix},
\end{equation}
while $X$ is resolved through
\begin{equation}\label{eq:w_hat_solve}
    \begin{cases}
        & X_{1}=-Y_{1}/h_{2,1},\ \ X_{{2}}=-\left(Y_2-\left(\mu-h_{2,2}\right)X_{1}\right)/h_{3,2}, \\
        & X_{i}=-\left(Y_{{i}}-\left(\mu-h_{i,i}\right)X_{{i-1}}+\sum\limits_{j=1}^{i-2}h_{j+1,i}X_{j}\right)/h_{i+1,i}, \ \ i=3,\dots,\gamma-1, \\
        & P^*(\mu)\ X_{{\gamma}}=\left(\Pi_{i=1}^{\gamma-1}h_{i+1,i}\right)\left(k_{\gamma}A_{\gamma}Y_{{\gamma}}-\sum\limits_{i=1}^{\gamma-1}l_{i,\gamma}X_{i}\right),
    \end{cases}
\end{equation}
where $\{l_{i,\gamma}\}_{i=1}^{\gamma-1} \subset \mathbb{C}^{n \times n}$ is defined in \eqref{eq:lll_def}. Crucially, the dominant cost in computing $Z$ in \eqref{eq:sii_def} reduces to solving $n$-dimensional systems for $X_{{\gamma}}$.

By combining trapezoidal quadrature and block LU decomposition, we develop an efficient method for computing $U = \rho(C_1^{-1}C_0)Y$ (Line 3, Algorithm \ref{alg:si_filter}), as outlined in Algorithm \ref{alg:rho_plus_Y}. Notably, Algorithm \ref{alg:si_filter} maintains a fixed quadrature order $k$ across iterations. This choice ensures that the set of linear systems $P^*(s^{(k)}_j)$ ($j = 1, \dots, k$) remains invariant, allowing their pre-factorization prior to subspace iteration{---}significantly accelerating linear system solves within each iteration.
Furthermore, Line 3 of Algorithm \ref{alg:rho_plus_Y} reveals that the computations of $G_j$ ($j= 1, \dots, k$) are fully independent. This property makes the approximation
$$\rho(C_1^{-1}C_0)Y \approx \sum_{j=1}^k g^{(k)}_j \left(s^{(k)}_j C_1 - C_0\right)^{-1} C_1 Y  $$
inherently parallelizable, enabling efficient exploitation of parallel computing resources.

\begin{algorithm}[h!!!]
\caption{A practical method for computing \(\rho(C_1^{-1}C_0)Y\)} \label{alg:rho_plus_Y}
\begin{algorithmic}[1]
\renewcommand{\algorithmicrequire}{\textbf{Input:}}
\renewcommand{\algorithmicensure}{\textbf{Output:}}
\REQUIRE The order \(k\) of the trapezoidal rule.
\ENSURE \(\zeta_{c,r,k}(C_1^{-1}C_0)Y=\sum\limits_{j=1}^{k}g^{(k)}_j\left(s^{(k)}_jC_1-C_0\right)^{-1}C_1Y \approx \rho(C_1^{-1}C_0)Y\).
\smallskip
\STATE Generate \(g^{(k)}_j\) and \( s^{(k)}_j,\ j=1,\dots,k\) based on \eqref{eq:pole_weight};
\STATE \textbf{for} \(j=1\) to \(k\) \textbf{do}
\STATE \ \ \ \ \ solve \(G_j=S^{\rm T}Z\) with \(\mu=s^{(k)}_j\) according to \eqref{eq:wk1_iterate} and \eqref{eq:w_hat_solve};
\STATE \textbf{end for}
\STATE \(\zeta_{c,r,k}(C_1^{-1}C_0)Y=S\sum\limits_{j=1}^kg^{(k)}_jG_j\).
\end{algorithmic}
\end{algorithm}

\subsection{The CORK framework}\label{subsec:cork}
Another useful approach for  {solving the} large-scale matrix pair $(C_0, C_1)$ in \eqref{eq:lin_important} is   the compact rational Krylov method (CORK) \cite{bemm:2015}. This approach centers on a matrix-valued function of the form
\begin{equation}\label{eq:cork_function}
\widetilde{P}(x) = \sum_{i=0}^{\gamma-1} \left(\widetilde{A}_i -x \widetilde{B}_i\right) \varphi_i(x), \quad \widetilde{A}_i, \widetilde{B}_i \in \mathbb{C}^{n \times n},
\end{equation}
where $\varphi_i : \mathbb{C} \to \mathbb{C}$ are polynomial or rational functions satisfying the linear relation
\begin{equation}
\left(M - xN\right) \varphi(x) = \bzs,
\end{equation}
with $\text{rank}(M - xN) = \gamma - 1$ for all $x \in \mathbb{C}$ and $\varphi(x) = [\varphi_0(x), \dots, \varphi_{\gamma-1}(x)]^\text{T} \neq \bzs$.
The CORK method computes the eigenvalues of the CORK linearization of $\widetilde{P}(x)$, given by
\begin{equation}\label{eq:cork_pencil}
    \mathscr{L}_{\widetilde{P}}(x)=\begin{bmatrix}
        \widetilde{A}_0-x\widetilde{B}_0&\cdots&\widetilde{A}_{\gamma-1}-x\widetilde{B}_{\gamma-1}\\\hline\\&\left(M-xN\right) \otimes I_n&
    \end{bmatrix},
\end{equation}
using a rational Krylov method that fully exploits its Kronecker structure. For detailed implementation strategies and eigenpair computations of \eqref{eq:cork_pencil}, we refer to \cite{bemm:2015}.

In our setting, by dividing the first block row of $(C_0,C_1)$ \eqref{eq:lin_important} by $-k_{\gamma-1}$ and applying the relation $h_{i+1,i} = \frac{k_{i-1}}{k_i}$, we obtain an equivalent representation of $P^*$ consistent with \eqref{eq:cork_function}. Specifically, the coefficients take the form:
\begin{equation}\label{eq:va_cork_coeffs_A}
    \widetilde{A}_i=A_{i}-\frac{h_{i+1,\gamma}}{h_{\gamma+1,\gamma}}A_{\gamma},\ \ i=0,1,\dots,\gamma-1,
\end{equation}
\begin{equation}\label{eq:va_cork_coeffs_B}
    \widetilde{B}_i=\begin{cases}
        0_{n \times n}, \ \ \ \ \ \ \ \ &i=0,1,\dots,\gamma-2, \\ -\frac{1}{h_{\gamma+1,\gamma}}A_{\gamma}, \ \ &i=\gamma-1.
    \end{cases}
\end{equation}
Furthermore, relying upon the recurrence relation \eqref{eq:recurrence_relation}, we derive 
\begin{equation}
    \left(H([\gamma],[\gamma-1])^{\rm T}-x\begin{bmatrix}
    I_{\gamma-1}&0_{(\gamma-1)\times 1}
    \end{bmatrix}\right)\begin{bmatrix}
        \vartheta_0(x) \\ \vdots \\ \vartheta_{\gamma-1}(x)
    \end{bmatrix} = \bzs.
\end{equation} 

\section{Numerical experiments}\label{sec:numerical}
In this section, we assess the numerical performance of our proposed algorithm through MATLAB R2024a implementations in double precision arithmetic. All experiments are conducted on a 15.3-inch MacBook Air equipped with an M2 chip and 8GB of memory. We test the method using some benchmark problems from \cite{guti:2017,hade:1967,saem:2020,tasa:2024,vafe:2014}.

To evaluate the effectiveness of our approach, we compare the computed eigenvalues with those obtained by several state-of-the-art methods:
\begin{itemize}
\item 
Beyn's method \cite{beyn:2012} (using the implementation from \cite[Figure 5.3]{guti:2017}),
\item 
RSI \cite{tasa:2024} (code available at \url{https://github.com/tang0389/RSI_NLEVP}),

\item Set-valued AAA \cite{limp:2022}, and
\item 
NLEIGS \cite{guvr:2014} (code available at \url{http://twr.cs.kuleuven.be/research/software/nleps/nleigs.html}).
\end{itemize}
To further demonstrate the advantages of rational filtering and CORK (implementation at \url{http://twr.cs.kuleuven.be/research/software/nleps/cork.html}), we include three large-scale NEPs. These test cases allow us to compare both accuracy and computational efficiency when computing eigenpairs using our approach versus MATLAB's built-in {\sf eig} function.

For clarity, we introduce the following nomenclature for our proposed algorithms.

\begin{description}
\item[\textsf{NEP\_MiniMax}:] It computes eigenpairs by 
\begin{enumerate}
\item obtaining a rational minimax  approximant  by \textsf{m-d-Lawson}   (Algorithm \ref{alg:Lawson}) and
\item solving the resulting matrix pencil eigenvalue problem using MATLAB's built-in \textsf{eig} function.
\end{enumerate}

\item[\textsf{NEP\_MiniMax-FILTER}:] It computes eigenpairs by 
\begin{enumerate}
\item constructing a rational minimax  approximant  via \textsf{m-d-Lawson} and
\item performing subspace iteration with rational filtering (Algorithm~\ref{alg:si_filter}) to solve the matrix pencil problem.
\end{enumerate}

\item[\textsf{NEP\_MiniMax-CORK}:] It computes eigenpairs by 
\begin{enumerate}
\item constructing a rational minimax  approximant  via \textsf{m-d-Lawson} and
\item applying the CORK method \cite{bemm:2015} to solve the resulting matrix pencil.
\end{enumerate}
\end{description}

\begin{example}\label{ex:example1}
Our first example is from \cite[Section 2.1]{guti:2017}  with 
    \begin{equation}\label{eq:example1}
        T(x)=\begin{bmatrix}
            {e}^{{\rm i}x^2}&1 \\ 1&1
        \end{bmatrix}.
    \end{equation} 
The matrix $T$ becomes singular precisely at the points satisfying $e^{ix^2} = 1$, which occur when $x \in \{\pm\sqrt{2\pi k} : k \in \mathbb{Z}\}$. Consequently, the eigenvalues of $T$ are given by
$$\Lambda^*:=\{\lambda_k = \pm\sqrt{2\pi k}, \quad k = 0, \pm1, \pm2, \dots\},$$
with $[1, -1]^\mathrm{T}$ serving as a corresponding right eigenvector for every eigenvalue. In our analysis, we focus on the disk $\Omega={\cal D}(0,3)$ centered at the origin ($c = 0$) with radius $r = 3$, which forms our region of interest.

Within our framework, we construct a set $\mathcal{X} = \{x_\ell\}_{\ell=1}^{100}$ consisting of 100 equidistant sampling points on the circular contour ${\cal C}(0,3)=\{x \in \mathbb{C} : |x| = 3\}$.
 Although the NEP in this example is only two-dimensional, uniformly approximating $\bt(x) = [e^{ix^2}, 1]^\mathrm{T}$ with high accuracy over the complex disk is not a low-degree task. Indeed, for \(x=3e^{{\rm i}\theta}\in\partial\Omega\),
$
|e^{{\rm i}x^2}|=e^{-9\sin(2\theta)}
$
 ranges from \(e^{-9}\) to \(e^9\) along the boundary. We therefore examine type \((d,d)\) approximations of $\bt(x)$ for \(0\le d\le30\) on $\mathcal{X}$ and evaluate the error $\epsilon_{{\boldsymbol{\xi}^*}}^{{\cal X}_{\rm val}}$ given in \eqref{eq:minimaxerr-val} on an independent validation set \({\cal X}_{\rm val}\) of 2000  equidistant boundary points. By applying {\sf m-d-Lawson} and set-valued AAA (SV-AAA)  {\cite{limp:2022}} to obtain their corresponding type $(d,d)$ rational approximants, we show in Figure~\ref{fig:err_curve_gauss} the errors $\epsilon_{{\boldsymbol{\xi}^*}}^{{\cal X}_{\rm val}}$. One can see that the associated  $\epsilon_{{\boldsymbol{\xi}^*}}^{{\cal X}_{\rm val}}$   
decreases   with \(d\) and reaches \(10^{-10}\) near \(d=28\).  

 \begin{figure}[h!!!]
        \centering
        \includegraphics[width=0.49  \textwidth]{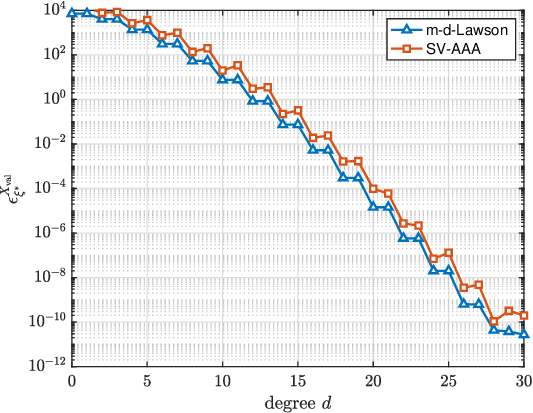}
        \caption{\footnotesize For types $(d,d)$  approximants ${\boldsymbol{\xi}^*}$ of $\bt(x) = [e^{ix^2}, 1]^\mathrm{T}$ over $\mathcal{X}$ produced by {\sf m-d-Lawson} and SV-AAA, it plots the error $\epsilon_{{\boldsymbol{\xi}^*}}^{{\cal X}_{\rm val}}$ defined in \eqref{eq:minimaxerr-val}.}
    \label{fig:err_curve_gauss}
    \end{figure}

Using {\sf NEP\_MiniMax} and   SV-AAA, let
$\Lambda_d$ denote the corresponding computed set of eigenvalues in $\Omega$ for type $(d,d)$ rational approximant.   
We measure the spectral error by
\begin{equation}\nonumber
    D_{\rm H}(\Lambda_d,\Lambda^*)=\max\left\{\max_{\lambda \in \Lambda_d}\min_{\mu \in \Lambda^*}\left|\lambda-\mu\right|,\max_{\mu \in \Lambda^*}\min_{\lambda \in \Lambda_d}\left|\lambda-\mu\right|\right\}.
\end{equation}
Table \ref{table:example1_degree} reports $\epsilon_{\xi^*}^{X_{\rm val}}$  and $D_{\rm H}\left(\Lambda_d,\Lambda^*\right)$ for  {\sf NEP\_MiniMax} and   SV-AAA.  One can see that {\sf NEP\_MiniMax} in general achieves more accurate than SV-AAA for these types.
\begin{table}[h!!!]
        \caption{$\epsilon_{\xi^*}^{X_{\rm val}}$  and $D_{\rm H}\left(\Lambda_d,\Lambda^*\right)$ for  {\sf NEP\_MiniMax} and   SV-AAA in Example~\ref{ex:example1}}
        \centering
        \begingroup
        \small
        \setlength{\tabcolsep}{4pt}
        \renewcommand{\arraystretch}{1.3}
        \begin{tabular}{|c|l|c|c|c|c|}
        \hline
        Metric & Method & type $(6,6)$ & type $(10,10)$ & type $(18,18)$ & type $(24,24)$ \\
        \hline
        \multirow{2}{*}{$\epsilon_{\xi^*}^{X_{\rm val}}$}
        & {\sf NEP\_MiniMax} & 3.1283e+02 & 7.6303e+00 & 3.0205e-04 & 2.0142e-08 \\
        \cline{2-6}
        & SV-AAA & 7.7101e+02 & 1.9856e+01 & 1.6420e-03 & 6.9928e-08 \\
        \hline
        \multirow{2}{*}{$D_{\rm H}\left(\Lambda_d,\Lambda^*\right)$}
        & {\sf NEP\_MiniMax} & 1.9927e+00 & 4.2152e-01 & 1.3169e-04 & 1.2231e-05 \\
        \cline{2-6}
        & SV-AAA & 1.9749e+00 & 9.5162e-01 & 4.2403e-03 & 4.2963e-05 \\
        \hline
        \end{tabular}
        \endgroup
        \label{table:example1_degree}
    \end{table}

For a particular type (28,28) computed by {\sf NEP\_MiniMax},   Figure \ref{fig:fig_example1} (right panel) further displays the zeros of $p_1^*(x)$, $p_2^*(x)$, and $q^*(x)$. We observed that   $\boldsymbol{\xi}^*(x)$ remains pole-free within our target domain $\Omega$.
Figure \ref{fig:fig_example1} (left panel) plots the computed eigenvalues of $T(x)$ inside $\Omega$ as blue dots. It is interesting to observe that there are additional 12 computed eigenvalues of $P^*(x)$ outside $\Omega$ (labelled as red $+$)   matching the exact ones in $\Lambda^*$. This occurs because the rational approximant $\boldsymbol{\xi}^*(x)$ captures the analytic continuation of $\boldsymbol{t}(x)$ beyond the original domain $\Omega$.  
    \begin{figure}[h!!!]
        \centering
        \includegraphics[width=0.9\textwidth]{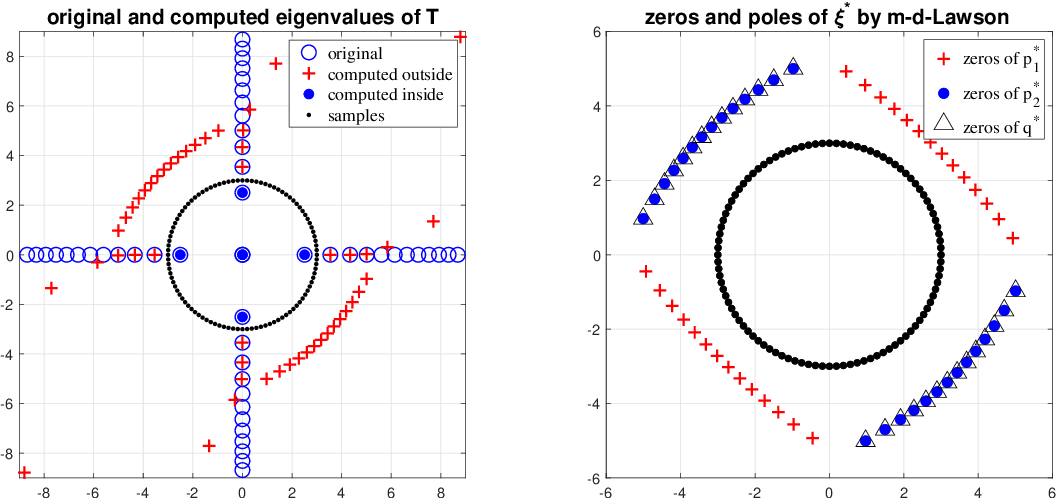}
        \caption{\footnotesize The left panel displays eigenvalue computations for \eqref{eq:example1}, where blue dot and red $+$ represent eigenvalues inside and outside the sampled region (black dots) respectively, obtained via Algorithm \ref{alg:linear_mdlawson}; original eigenvalues from \eqref{eq:example1} are shown as blue circles for comparison. The right panel plots the poles and zeros of ${\boldsymbol{\xi}}^*$ of type $(28,28)$, demonstrating that {\sf m-d-Lawson} produces a numerically minimax rational approximation, notably pole-free within $\Omega$.}
    \label{fig:fig_example1}
    \end{figure}

To numerically assess Remark~\ref{rmk:numerical_eta_estimate}, for $\boldsymbol{\xi}^*(x)$ of type $(28,28)$, we evaluate~\eqref{eq:eta_discrete_apriori_est} on the validation set ${\cal X}_{\rm val}$. The resulting quantities are
\[
\epsilon_{{\boldsymbol{\xi}^*}}^{{\cal X}_{\rm val}}=3.7300\times10^{-11},\qquad
\alpha_{{\cal X}_{\rm val}}=1.6232,\qquad
\sqrt{\left\|\mathscr{G}_{\{E_i\}_i}\right\|_2}=1.7321,
\]
and hence
$
\widehat{\nu}_{\Gamma}^{\,{\rm bd}}=1.0487\times10^{-10}<1.
$
Together with the pole-free property, this value provides a numerical evidence for the fulfillment of the sufficient condition in Proposition~\ref{prop:spectral_count_preservation}. 
\end{example}

\begin{example}\label{ex:example3}
     {(Time delay)}. We consider the example   in \cite[Example 1]{saem:2020} where 
    \begin{equation}\label{eq:time_delay2_def}
        T(x)=-B_0+xI+{ e}^{-\tau x}A_1,
    \end{equation}
and 
       $ B_0=\begin{bmatrix}
            -5&1\\2&-6
        \end{bmatrix},\ A_1=\begin{bmatrix}
            2&-1\\-4&1
        \end{bmatrix},\ \tau=1.$
    This problem is also contained in the NLEVP collection \cite{behm:2013} under the name {\sf time}\_{\sf delay2}. 

   In this example, we use {\sf NEP\_MiniMax} to compute eigenvalues of $T(x)$ within the disk ${\cal D}(-1,6)$ (centered at $c=-1$ with radius $r=6$). For the approximation, we collect 50 equiangular samples on  ${\cal C}(-1,6)$ and construct the set $\mathcal{X}$. Applying {\sf m-d-Lawson}, we obtain a type (10,10) vector-valued  rational approximant  $\boldsymbol{\xi}^*(x) = \frac{1}{q^*(x)}[p^*_1(x), p^*_2(x), p^*_3(x)]^{\T}$ to $\bt(x) = [1, x, e^{-x}]^{\T}$, achieving an approximation error $\sqrt{e(\boldsymbol{\xi}^*)} < 10^{-7}$.
The right panel of Figure \ref{fig:fig_example3} plots the zeros and poles of $\boldsymbol{\xi}^*(x)$, confirming its pole-free property within the disk. Also, we compute eigenvalues of the rational matrix $R^*(x) = \frac{1}{q^*(x)}(-p^*_1(x)B_0 + p^*_2(x) I_2 + p^*_3(x) A_1)$ via linearization, yielding eigenvalue approximations for $T(x)$, as shown in the left panel of Figure \ref{fig:fig_example3}.

\begin{figure}[h!!!]
        \centering
         \includegraphics[width=0.9\textwidth]{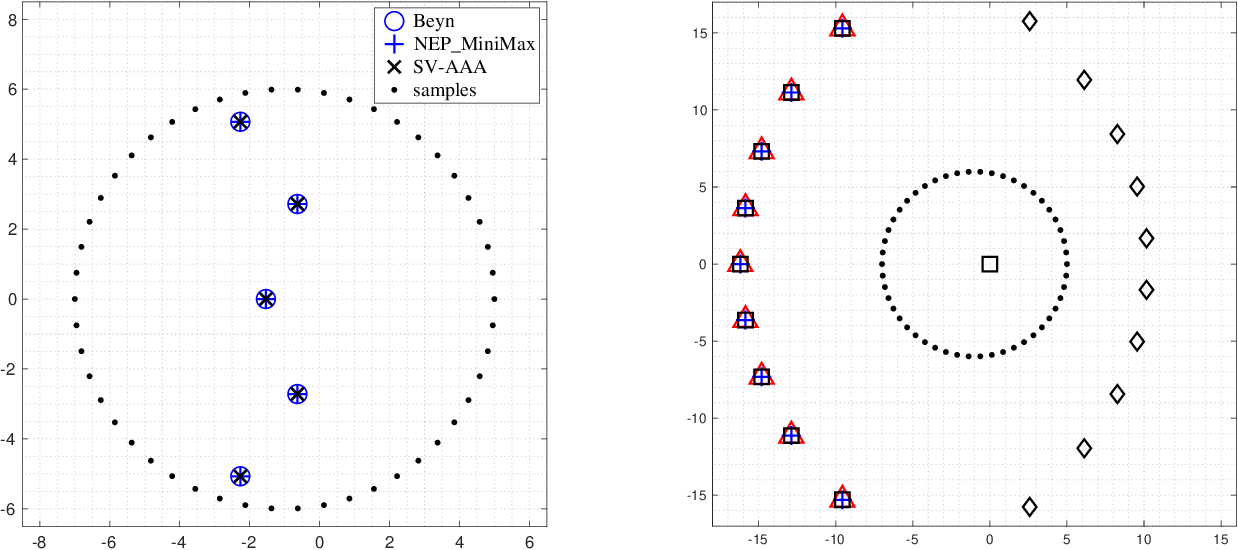}
        \caption{\footnotesize Left panel: Computed eigenvalues of \eqref{eq:time_delay2_def} obtained via {\sf NEP\_MiniMax} ($+$), Beyn's method ($\circ$), and SV-AAA ($\times$).
Right panel: Illustration of poles of $\boldsymbol{\xi}^*$ ($\bigtriangleup$) and zeros of $p^*_1$ ($+$), $p^*_2$ ($\square$), and $p^*_3$ ($\Diamond$), computed using {\sf m-d-Lawson}.  Only points with magnitude less than 18 are displayed.}
    \label{fig:fig_example3}
    \end{figure}

    For comparison, we compute eigenvalues within ${\cal D}(-1,6)$ using Beyn's integral method  {\cite{beyn:2012}} and  SV-AAA. In Beyn's method, we use the numerical quadrature nodes $\mathcal{X} $ (i.e., $50$ equispaced points in  {the trapezoidal rule}) as moment-computation points and fix the probing matrix to a single column (see \cite[Section 5]{guti:2017} for details). For SV-AAA, we increase the boundary sampling density by a factor of two and iteratively select 11 support points (corresponding to a type $(10,10)$ rational approximation for $T(x)$) to ensure an approximation error below $10^{-8}$. Table \ref{table:table1} summarizes the   residuals  \begin{equation}\label{eq:relative_residual_def}
\epsilon(\lambda, \bu) =  \frac{\|T(\lambda)\bu\|_2}{\|\bu\|_2} 
\end{equation}  
 of the eigenpairs obtained via {\sf NEP\_MiniMax}, Beyn's method, and SV-AAA.

    \begin{table}[h!!!]
        \caption{Residuals $\epsilon(\lambda, \boldsymbol{u})$ for Example \ref{ex:example3}}
        \begin{center}\resizebox{150mm}{12mm}{\tabcolsep0.15in
        \begin{tabular}{|c|c|c|c|c|c|c|}\hline
        relative residuals &eig-index  &eig-index &  eig-index  &  eig-index  &  eig-index \\
       $\epsilon(\lambda, \boldsymbol{u})$ & 1  & 2  &   3   &  4 & 5 \\
        \hline
        \hline
      {\rm {\sf NEP\_MiniMax}}   & 2.3660e-09 &  3.1122e-11 &  5.2727e-10  & 8.3716e-11  & 2.8457e-09 \\  
         {\rm Beyn's\ method}   & 4.9460e-06 &  3.0185e-04 &  3.9748e-05 &  3.0185e-04 &  4.9460e-06 \\  
     {\rm SV-AAA}   &2.4814e-09 & 1.4055e-09 &  1.9067e-10 &  4.1219e-09 &  7.6956e-09   \\
    \hline
    \end{tabular}
    }
    \end{center}
    \label{table:table1}
    \end{table}

To numerically verify Proposition~\ref{prop:spectral_count_preservation}, we likewise evaluate \eqref{eq:eta_discrete_apriori_est} on an independent validation set ${\cal X}_{\rm val}$ of 2000 equidistant points on ${\cal C}(-1,6)$. We have
\[
\epsilon_{{\boldsymbol{\xi}^*}}^{{\cal X}_{\rm val}}=5.0729\times10^{-8},\qquad
\alpha_{{\cal X}_{\rm val}}=0.26072,\qquad
\sqrt{\left\|\mathscr{G}_{\{E_i\}_i}\right\|_2}=8.8854,
\]
and hence
$
\widehat{\nu}_{\Gamma}^{\,{\rm bd}}=1.1752\times10^{-7}<1.
$
As all computed poles are outside $\Omega$, this value lends numerical support to the sufficient condition stated in Proposition ~\ref{prop:spectral_count_preservation}. 
\end{example}
 
\begin{example}\label{ex:hadeler}
    (Hadeler problem). We next consider the Hadeler problem  \cite{eims:2020,hade:1967} where
    \begin{equation}\label{eq:hadeler_def}
        T(x)=\left({  e}^{x}-1\right)B_1+x^2B_2-B_0,
    \end{equation}
    with the coefficient matrices
    \begin{equation*}
        B_0=b_0I_n,\ \ B_1=[b^{(1)}_{ij}],\ \ B_2 = [b^{(2)}_{ij}],
    \end{equation*}
    \begin{equation*}
        b^{(1)}_{ij}=(n+1-\max\{i,j\})ij, \ \ b^{(2)}_{ij}=n\delta_{ij}+\frac{1}{i+j},
    \end{equation*}
    and  \(n=200\) and  \(b_0=100\).

The eigenvalues of \eqref{eq:hadeler_def} are all real, consisting of $n$ negative and $n$ positive values. These eigenvalues become more uniformly spaced farther from the origin, with the smallest eigenvalue near $-48$. Following \cite{saem:2020}, we compute the eigenvalues within the circular region $\Omega= {\cal D}(-30, 11.5)$.
We discretize the boundary ${\cal C}(-30, 11.5)$ using 50 equidistant points to form the approximation set $\mathcal{X}$. Applying {\sf m-d-Lawson}, we obtain a type $(6,6)$ vector-valued rational approximant $\boldsymbol{\xi}^*(x)$ for $\bt(x)=[-1,x^2, e^{x} - 1]^{\T}$, achieving $\sqrt{e(\boldsymbol{\xi}^*)} < 10^{-10}$. Crucially, $\boldsymbol{\xi}^*(x)$ remains pole-free in $ {\cal D}(-30, 11.5)$. Finally, we solve the resulting $1200 \times 1200$ pencil $C_0 - \lambda C_1$ to compute the eigenvalues.
    
    \begin{figure}[h!!!]
        \centering
        \includegraphics[width=0.95\textwidth]{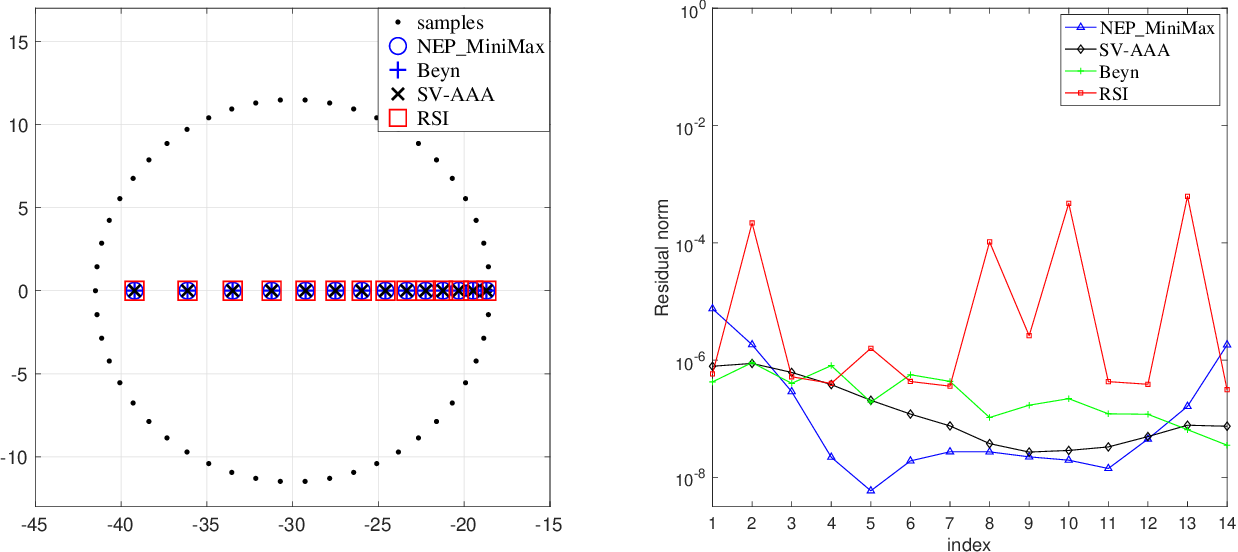}
        \caption{\footnotesize  Left: Eigenvalues of \eqref{eq:hadeler_def} computed via \textsf{NEP\_MiniMax} ($\circ$), Beyn's method ($+$), SV-AAA ($\times$), and RSI ($\square$) are displayed. Right: Residuals $\epsilon(\lambda,\boldsymbol{u})$ of the eigenpairs obtained from \textsf{NEP\_MiniMax}, Beyn's method, SV-AAA, and RSI in Example~\ref{ex:hadeler}.}
    \label{fig:fig_hadeler1}
    \end{figure}

    To facilitate comparison with RSI~\cite{saem:2020}, Beyn's method, and SV-AAA \cite{limp:2022}, we begin by computing eigenvalues using \textsf{NEP\_MiniMax}. Figure~\ref{fig:fig_hadeler1} displays these numerical results together with their corresponding   residuals $\epsilon(\lambda, \boldsymbol{u})$ in \eqref{eq:relative_residual_def}  for eigenpairs located within $\Omega$.
We observe that the   residuals $\epsilon(\lambda,\boldsymbol{u})$  can be improved by increasing the approximation degree in \textsf{m-d-Lawson}. Higher degrees yield smaller $\sqrt{e(\boldsymbol{\xi}^*)}$, which, according to Corollary~\ref{coro:coro0}, further reduces $\epsilon(\lambda,\boldsymbol{u})$. 
To numerically validate the estimation formula~\eqref{eq:err_approx}, Figure~\ref{fig:fig_hadeler_prio_real} (left) compares 
\begin{itemize}
    \item the actual   residuals from \textsf{NEP\_MiniMax} using rational approximations of degrees $6$, $7$, and $8$, 
    \item their corresponding a priori estimates $\sqrt{\left\|\mathscr{G}_{\{E_i\}_i}\right\|_2 \cdot e(\boldsymbol{\xi}^*)}$, where $\sqrt{\left\|\mathscr{G}_{\{E_i\}_i}\right\|_2} = 1.0282 \times 10^8$.
\end{itemize}
The results confirm the validity of Corollary~\ref{coro:coro0} and demonstrate that increasing the approximation degree consistently reduces the   residuals.

    \begin{figure}[h!!!]
        \centering
        \includegraphics[width=0.9\textwidth]{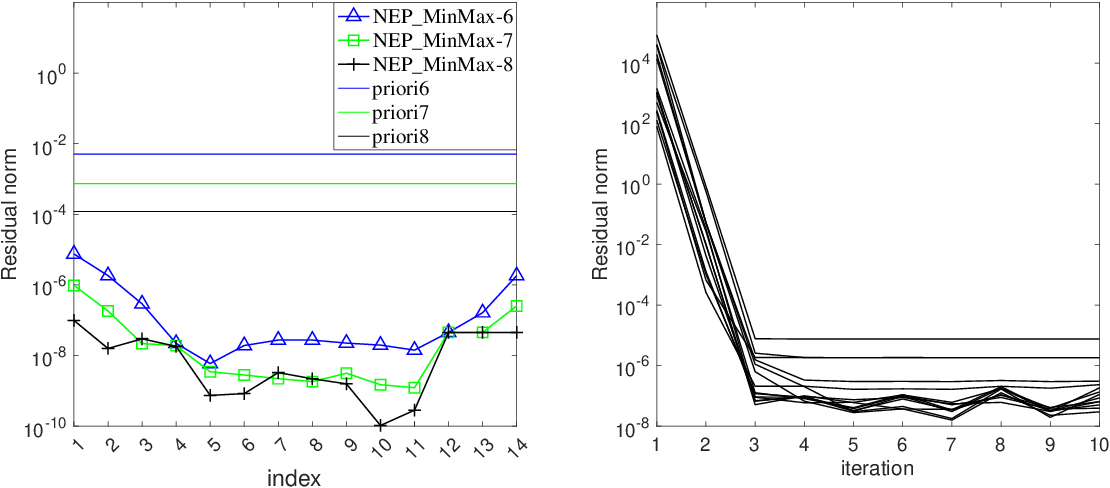}
        \caption{\footnotesize Left: Residuals $\epsilon(\lambda, \boldsymbol{u})$ of the eigenpairs computed by {\sf NEP\_MiniMax} using rational approximations of degrees 6, 7, and 8, along with their corresponding a priori estimates $\sqrt{\|\mathscr{G}_{\{E_i\}_i}\|_2 \, e(\boldsymbol{\xi}^*)}$. Here, {\sf NEP\_MiniMax}-6 refers to the computation where {\sf m-d-Lawson} produces a type (6,6) rational approximation, followed by linearization and eigenvalue extraction (priori6 denotes the a priori estimate  for {\sf NEP\_MiniMax}-6; analogous conventions apply to degrees 7 and 8).
Right: Residuals $\epsilon(\lambda, \boldsymbol{u})$ for eigenpairs in $\Omega$ over the first 10 iterations of {\sf NEP\_MiniMax-FILTER}-6.}
    \label{fig:fig_hadeler_prio_real}
    \end{figure}

In this example, we evaluate the performance of \textsf{NEP\_MiniMax-FILTER} for comparison, as the pencil's non-trivial scale warrants such analysis. Following the strategy in~\cite{chli:2025} (adapted for the original NEP via Remark~\ref{rmk:nep_pencil_relation}), we verify the convergence of rational-filter subspace iteration (Algorithm~\ref{alg:si_filter}) using two thresholds $0 < \tau_r < \tau_g$. An eigenpair $(\lambda,\boldsymbol{v})$ of $(C_0,C_1)$ with relative residual $\varsigma_{c,r}(\lambda,\boldsymbol{v}) \geq \tau_g$ is classified as a ghost eigenpair, where
$$\varsigma_{c,r}(\lambda,\boldsymbol{v}) := \frac{\|T(\lambda)\boldsymbol{v}_{1}\|_2}{(|c|+r)\|\boldsymbol{v}_{1}\|_2},~c=-30, ~r=11.5,$$
and $\boldsymbol{v}_{1}$ comprises the first $n$ rows of $\boldsymbol{v}$. For this example, we set $\tau_r = 10^{-4}$ and $\tau_g = 10^{-2}$. The right panel of Figure~\ref{fig:fig_hadeler_prio_real} displays the residuals $\epsilon(\lambda,\boldsymbol{v}_{1}) = (|c|+r)\varsigma_{c,r}(\lambda,\boldsymbol{v})$ for eigenpairs in $\Omega$ during the first 10 iterations of the subspace method with a degree 6   filter. The results show that \textsf{NEP\_MiniMax-FILTER} attains accuracy comparable to \textsf{NEP\_MiniMax} within just 5 iterations. Additionally, experiments reveal a significant runtime (including the time of {\sf m-d-Lawson}) reduction from 12 seconds (\textsf{NEP\_MiniMax}) to 3 seconds for this case.
 
Following the previous examples, we evaluate \eqref{eq:eta_discrete_apriori_est} on a  set $ X_{\text{val}} $ consisting of 2000 equidistant points on ${\cal C}(-30, 11.5) $.  This gives 
\[
\epsilon_{{\boldsymbol{\xi}^*}}^{{\cal X}_{\rm val}}=4.9440\times10^{-11},\qquad
\alpha_{{\cal X}_{\rm val}}=6.4084\times10^{-4},\qquad
\sqrt{\left\|\mathscr{G}_{\{E_i\}_i}\right\|_2}=1.0282\times10^{8},
\]
and hence
$
\widehat{\nu}_{\Gamma}^{\,{\rm bd}}=3.2578\times10^{-6}<1.
$
\end{example}

\begin{example}\label{ex:gun}
     {(Gun problem)}. 
     We next examine the ``gun" problem, a relatively large-scale model of a radio-frequency gun cavity described by the nonlinear eigenvalue problem with
\begin{equation}\label{eq:gun_def}
T(x) = K - x M + {\rm i}\sqrt{x-\sigma_1^2}W_1 + {\rm i}\sqrt{x-\sigma_2^2}W_2,
\end{equation}
where $M$, $K$, $W_1$, and $W_2$ are symmetric $9956\times9956$ matrices from \cite{guti:2017,libl:2010}. The complex square root $\sqrt{\cdot}$ uses the principal branch. Following \cite{behm:2013,libl:2010}, we set $\sigma_1 = 0$ and $\sigma_2 = 108.8774$. Our computational target consists of the 20 eigenvalues nearest to $250^2$ that lie within the upper semi-disk $\Omega$ centered at $c=250^2$ with radius $r=300^2 - 200^2$.
 
 We assess the performance of {\sf NEP\_MiniMax-CORK} in this example. Following a similar sampling construction in \cite{limp:2022},  we construct our approximation set $\mathcal{X}$ with 1025 points through a two-step process: first, we generate 525 equiangularly distributed boundary samples along $\partial\Omega$ (including  the interval $[c-r,c+r]$) where all sampling points correspond to uniformly spaced central angles; subsequently, we augment these boundary points with 500 randomly distributed interior points within the upper semi-disk $\Omega$ to ensure robust domain representation. Applying {\sf m-d-Lawson}, we then generate a type (6,6) rational approximation $\boldsymbol{\xi}^*(x)=[r_1^*(x),r_2^*(x),r_3^*(x),r_4^*(x)]^{\T}$ for $\bt(x)=[1,-x,{\rm i}\sqrt{x},{\rm i}\sqrt{x-\sigma_2^2}]^{\T}$. This approximation yields a structured  {$(6\cdot9956)\times(6\cdot9956)$ matrix-pencil} linearization $C_0 - \lambda C_1$ via Kronecker products, whose eigenvalues within $\Omega$ coincide with those of $R^*(x) = r_1^*(x) K + r_2^*(x) M + r_3^*(x) W_1 + r_4^*(x) W_2$, provided $\boldsymbol{\xi}^*(x)$ has no poles in $\Omega$. The CORK framework  is then used to solve the linearized problem (see \cite{bemm:2015} for implementation specifics).
Figure \ref{fig:gun_eigs} displays the poles of $\boldsymbol{\xi}^*(x)$ alongside the eigenvalues of $T(x)$ computed via {\sf NEP\_MiniMax-CORK}   within $\Omega$, compared against NLEIGS results. Notably, {\sf NEP\_MiniMax-CORK}  terminates in under 4 seconds (including the time of {\sf m-d-Lawson}), demonstrating its efficiency in  {computing the targeted eigenvalues of the gun problem from a split coefficient-vector rational surrogate}.

    \begin{figure}[h!!!]
        \centering
        \includegraphics[width=0.8\textwidth]{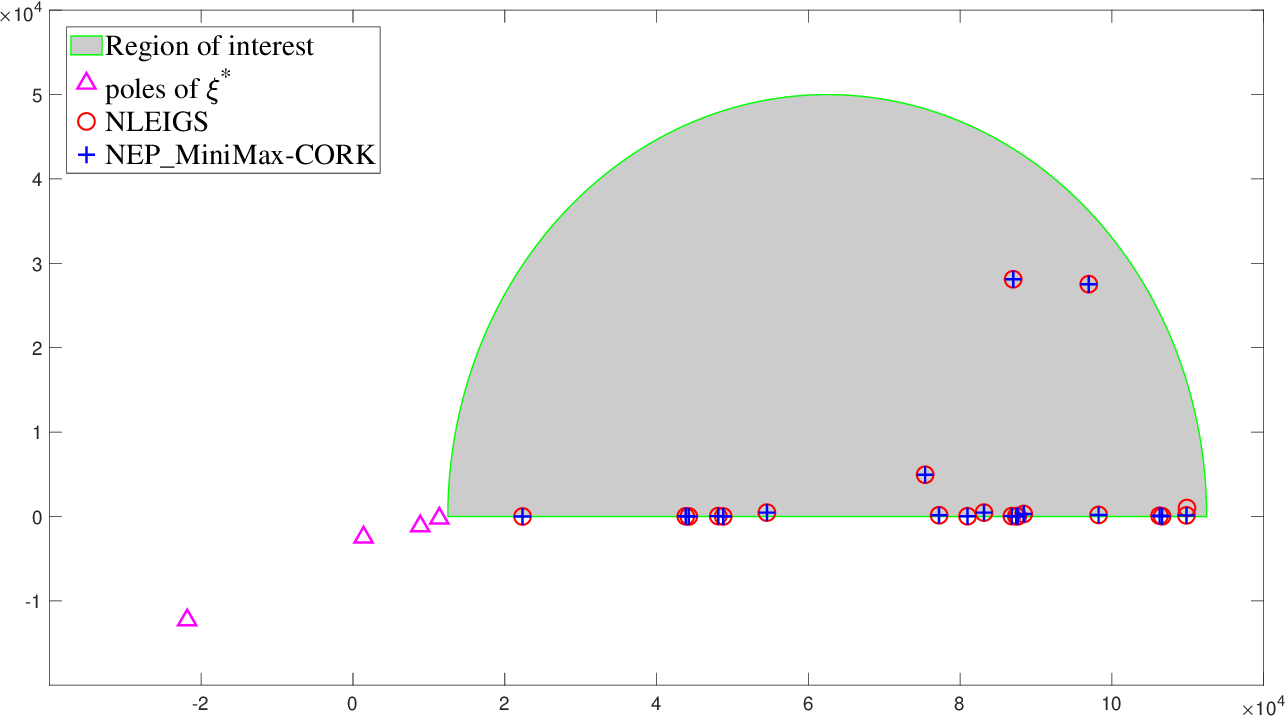}
        \caption{\footnotesize This figure displays the computed eigenvalues of the gun problem~\eqref{eq:gun_def} within the region of interest, obtained via NLEIGS (marked with red $\circ$) and \textsf{NEP\_MiniMax-CORK} (shown as blue $+$), along with the poles of $\boldsymbol{\xi}^*(\lambda)$ (indicated by $\bigtriangleup$). The results confirm that $\boldsymbol{\xi}^*(\lambda)$ remains pole-free throughout the region of interest.}
    \label{fig:gun_eigs}
    \end{figure}
    
 Figure~\ref{fig:gun_res_history} presents the convergence history of the normalized residual  
 \begin{equation}\label{eq:def_res_gun}
    \varrho_{\mathrm{gun}}(\lambda,\boldsymbol{u}) = \frac{\|T(\lambda)\boldsymbol{u}\|_2}{\|K\|_1 + |\lambda|\cdot\|M\|_1 + \sqrt{|\lambda-\sigma_1^2|}\cdot \|W_1\|_1 + \sqrt{|\lambda-\sigma_2^2|}\cdot\|W_2\|_1},
\end{equation}
 for eigenpairs of $T(x)$~\eqref{eq:gun_def} computed via \textsf{NEP\_MiniMax-CORK}. The implementation uses 
\begin{itemize}
    \item maximum subspace dimension $ 50$,
    \item $p = 35$ selected Ritz values per restart, and
    \item cyclically repeated shifts in rational Krylov steps.
\end{itemize}
The results show that all $20$ eigenvalues closest to $250^2$ were computed to a tolerance of $10^{-7}$ in fewer than 80 iterations (4 seconds, including the runtime of {\sf m-d-Lawson}).

    \begin{figure}[h!!!]
        \centering
         \includegraphics[width=0.8\textwidth]{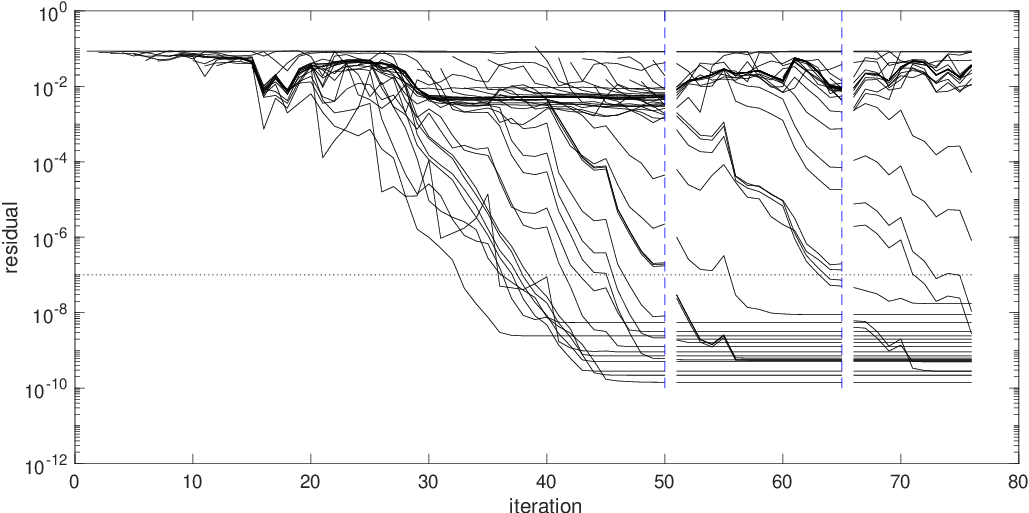}
        \caption{\footnotesize The convergence history of the residuals \(\varrho_{{\rm gun}}(\lambda,\pmb{u})\) \eqref{eq:def_res_gun} of the computed eigenpairs to \(T(x)\) \eqref{eq:gun_def} via \textsf{NEP\_MiniMax-CORK}  with maximum subspace dimension \(50\) and \(35\) selected Ritz values in every restart.}
    \label{fig:gun_res_history}
    \end{figure}

    \begin{figure}[h!!!]
        \centering
        \includegraphics[width=0.93\textwidth]{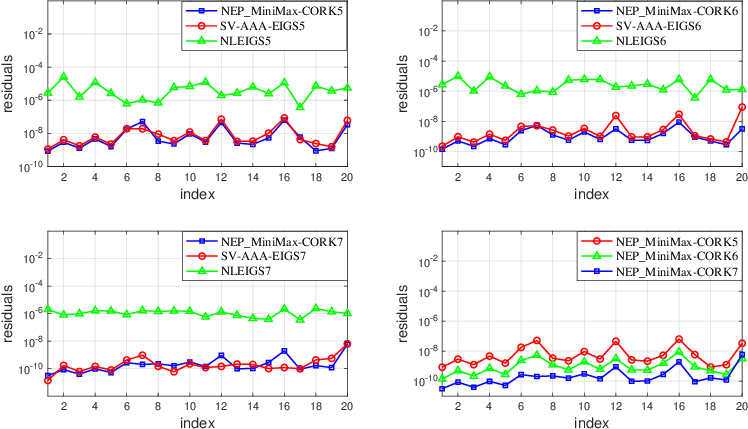}
        \caption{\footnotesize The residuals $\varrho_{\mathrm{gun}}(\lambda,\boldsymbol{u})$ of the computed eigenpairs are shown in the top-left, bottom-left, and top-right panels for \textsf{NEP\_MiniMax-CORK}, SV-AAA-EIGS, and NLEIGS at approximation degrees 5, 6, and 7 respectively. The bottom-right panel demonstrates how increasing the approximation degree further reduces the residuals obtained via \textsf{NEP\_MiniMax}.}
    \label{fig:fig_gun_apx_VS}
    \end{figure}
    
We present a comparative analysis of eigenvalue computation methods by evaluating the residuals $\varrho_{\mathrm{gun}}(\lambda,\boldsymbol{u})$~\eqref{eq:def_res_gun} for 20 eigenpairs obtained from \textsf{NEP\_MiniMax}, SV-AAA-EIGS, and NLEIGS at approximation degrees 5, 6, and 7. 
Figure~\ref{fig:fig_gun_apx_VS} demonstrates that \textsf{NEP\_MiniMax} achieves smaller residuals than NLEIGS (comparable to SV-AAA-EIGS) at identical approximation degrees. The figure's lower-right panel shows that increasing the approximation degree further reduces \textsf{NEP\_MiniMax}'s residuals.

{
To check Remark \ref{rmk:numerical_eta_estimate}, if we   evaluate~\eqref{eq:eta_discrete_apriori_est} on a  set ${\cal X}_{\rm val}$ containing 100 midpoint samples on the semicircular arc and 100 on the diameter, we find that 
\[
\epsilon_{{\boldsymbol{\xi}^*}}^{{\cal X}_{\rm val}}
 =5.4571\times10^{-3},\qquad
\alpha_{{\cal X}_{\rm val}}
 =1.1528\times10^{1},\qquad
\sqrt{\left\|\mathscr{G}_{\{E_i\}_i}\right\|_2}
 =1.2748\times10^{6},
\]
which unfortunately gives
$
\widehat{\nu}_{\Gamma}^{\,{\rm bd}}
 =8.0191\times10^{4}>1.
$
However, we can refine the estimate by balancing the split. In particular, 
we consider  
\[
\widetilde t_j(x)=s_jt_j(x),\qquad
\widetilde E_j=E_j/s_j,\qquad
s_j=\frac{\|E_j\|_{\F}}{\max_k\|E_k\|_{\F}},
\]
which leaves $T(x)=\sum_j t_j(x)E_j$ unchanged. Recomputing the type $(6,6)$ approximant yields
\[
\epsilon_{{\boldsymbol{\xi}^*}}^{{\cal X}_{\rm val}}
 =5.4906\times10^{-8},\qquad
\sqrt{\left\|\mathscr{G}_{\{\widetilde E_i\}_i}\right\|_2}
 =1.6191\times10^{6},
\]
and hence $\widehat{\nu}_{\Gamma}^{\,{\rm bd}}=1.0248>1$. While certification is still not obtained, balancing reduces the estimate by approximately $7.8\times10^4$. Moreover, neither approximant has a computed pole in $\bar{\Omega}$. These results demonstrate that the split representation can substantially affect both the coefficient-space approximation and the sharpness of the bound, motivating a further topic of   scaling balance of $T(x)$ in {\sf NEP\_MiniMax} for large-scale problems.
}
\end{example}

\begin{example}\label{ex:sandwich_beam}
    (Sandwich beam problem). The sandwich beam problem from the NLEVP
    collection~\cite{behm:2013} has the form 
    \begin{equation}\label{eq:sandwich_def}
        T(x)=K-x^2M+g(x)C,
        \qquad
        g(x)=
        \frac{G_0+G_\infty({\rm i}x\tau)^\alpha}
        {1+({\rm i}x\tau)^\alpha},
    \end{equation}
    where $K,M,C\in\mathbb{R}^{168\times168}$. We use NLEVP implementation with
    $G_0=3.504\times10^5$, $G_\infty=3.062\times10^9$,
    $\alpha=0.675$, and $\tau=8.23\times10^{-9}$, and take the fractional power in~\eqref{eq:sandwich_def}
 on its principal branch.

    Following~\cite{limp:2022,bemm:2013}, we consider the real interval $[200,30000]$ embedded in the disk $\Omega=\mathcal{D}(15100,14900)$ with boundary $\Gamma=\partial\Omega=\mathcal{C}(15100,14900)$. Using the affine mapping   $x=\phi(z):=15100+14900z$ for $z\in\overline{\widehat{\Omega}}$, where $\widehat{\Omega}:=\mathcal{D}(0,1)$ and $\widehat{\Gamma}:=\partial\widehat{\Omega}=\mathcal{C}(0,1)$, we transfer computations to the standard unit disk. With $\operatorname{Re}(x)\geq200$ on $\bar{\Omega}$, the selected branch of $g$ remains analytic in a neighborhood of $\bar{\Omega}$. All approximations are performed in the $z$-plane, while eigenvalues and residuals are evaluated and presented back in the original $x$-plane.

The approximation set consists of $100$ equiangular points on
    $\widehat\Gamma$, and we approximate
    $\bt(\phi(z))=[1,-\phi(z)^2,g(\phi(z))]^{\T}$. For validation we use
    $2000$ midpoint boundary samples and $2000$ area-uniform interior samples,
    denoted collectively by ${\cal X}_{\rm test}$.  For
    $\widehat R^*(z):=R^*(\phi(z))$, we report
    \begin{equation}\label{eq:sandwich_validation_error}
        \epsilon_{\rm mat}^{{\cal X}_{\rm test}}=
        \max_{z\in{\cal X}_{\rm test}}
        \frac{\|T(\phi(z))-\widehat R^*(z)\|_{\rm F}}
             {\|T(\phi(z))\|_{\rm F}},
    \end{equation}
    and, similarly to \eqref{eq:def_res_gun}, measure each computed eigenpair  by
    \begin{equation}\label{eq:sandwich_residual}
        \varrho_{\rm rel}(\lambda,\boldsymbol{u})=
        \frac{\|T(\lambda)\boldsymbol{u}\|_2}
             {\|T(\lambda)\|_1\|\boldsymbol{u}\|_2}.
    \end{equation}

    A type $(4,4)$ \textsf{m-d-Lawson} approximant is pole-free in
    $\overline{\widehat\Omega}$, has four finite poles, and gives
    $\epsilon_{\rm mat}^{{\cal X}_{\rm test}}=6.7185\times10^{-9}$. The
    resulting structured linearization has order $672$. We solve it
    by CORK on the upper half of the unit disk, using maximum subspace
    dimension $50$, $35$ retained Ritz values, and the ten cyclic shifts
    \[
        \mu_j=0.4\exp\!\left(\frac{(j-1/2)\pi{\rm i}}{10}\right),
        \qquad j=1,\ldots,10.
    \]
    After $95$ Krylov steps, CORK returns ten eigenpairs in
    $\Omega_+:=\{\lambda\in\Omega:\operatorname{Im}(\lambda)\geq0\}$, with
    largest residual $2.4678\times10^{-13}$. Figure~\ref{fig:sandwich_spectrum}
    shows these eigenvalues, the boundary samples, the nearest finite
      poles, and the reference values from~\cite{bemm:2013}.

 \begin{figure}[H]
        \centering
        \includegraphics[width=0.6 \textwidth]{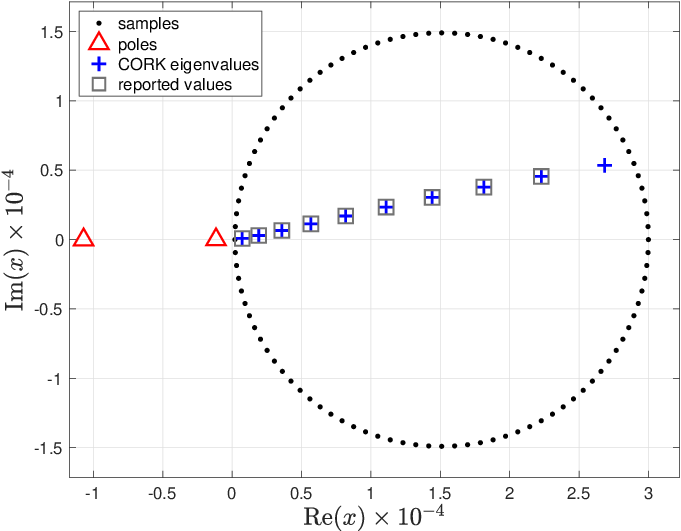}
        \caption{\footnotesize It shows samples used by the type
        $(4,4)$
        \textsf{m-d-Lawson}, the two finite   poles
        nearby, the ten eigenvalues computed by {\sf NEP\_MiniMax-CORK}, and the
        nine values reported in~\cite{bemm:2013} in the original
        $x$-plane.}
        \label{fig:sandwich_spectrum}
    \end{figure}
    
     To compare the rational approximation stages independently of Krylov
    parameters, we also solve the complete linearizations obtained from
    \textsf{NEP\_MiniMax} and SV-AAA~\cite{limp:2022} with 
    $2,\ldots,10$ poles, using the same map, split functions, training set, and
    test set. None of the resulting approximants has a pole in the closed
    target disk. Table~\ref{tab:sandwich_comparison} reports three
    representative cases. 

    \begin{table}[htbp]
        \centering
        \small
        \caption{Numerical results for sandwich-beam problems in Example \ref{ex:sandwich_beam}}
        \label{tab:sandwich_comparison}
        \begin{tabular}{@{}lccccc@{}}
            \hline
            Method & $d$ & $\epsilon_{\rm mat}^{{\cal X}_{\rm test}}$ & linearization size&
           \# of  eigenvalues & $\max\varrho_{\rm rel}$ \\ \hline
            \textsf{NEP\_MiniMax}  
                & 4 & $6.7185\times10^{-9}$ & 672 & 10 & $1.8558\times10^{-12}$ \\
            SV-AAA  
                & 4 & $8.0653\times10^{-7}$ & 1008 & 10 & $2.4010\times10^{-10}$ \\
            SV-AAA  
                & 7 & $1.1502\times10^{-9}$ & 1512 & 10 & $5.1030\times10^{-12}$ \\ \hline
        \end{tabular}
    \end{table}

    Figure~\ref{fig:sandwich_budget} provides the behavior  of the methods \textsf{NEP\_MiniMax} and SV-AAA with varying number of poles in their rational approximations. Using type $(4,4)$, \textsf{NEP\_MiniMax} reduces  $\epsilon_{\rm mat}^{{\cal X}_{\rm test}}$ by
    about two orders of magnitude compared with SV-AAA. SV-AAA becomes more
    accurate in validation only after using a larger linearization, and for
    the cases with more than $8$ poles, the two methods are comparable. The type $(4,4)$
    \textsf{NEP\_MiniMax} approximation is used as the representative setting
    because it is pole-free, returns all ten reported eigenpairs, and already
    yields residuals of order $10^{-12}$. The AAA-CORK experiment in~\cite{limp:2022}
    is not directly identical, since it keeps $K-x^2M$ exact and approximates
    only $g(x)$ from $10^4$ real samples.

    \begin{figure}[H]
        \centering
        \includegraphics[width=0.9\textwidth]{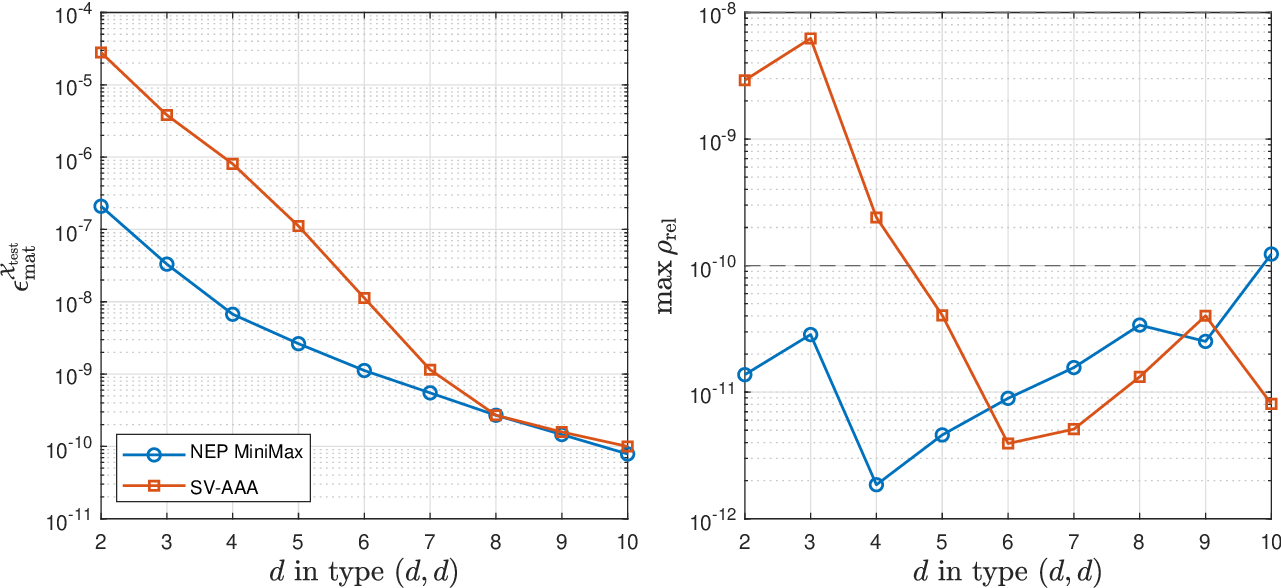}
        \caption{\footnotesize Behaviors of $\epsilon_{\rm mat}^{{\cal X}_{\rm test}}$ and $\varrho_{\rm rel}$  with various approximation types $(d,d)$  for the sandwich-beam problem on the
        same approximation set ${\cal X}$ and test set ${\cal X}_{\rm test}$.
        Left:   $\epsilon_{\rm mat}^{{\cal X}_{\rm test}}$ on ${\cal X}_{\rm test}$. Right: maximum  $\varrho_{\rm rel}$ of the computed eigenpairs
        in $\Omega_+$.}
        \label{fig:sandwich_budget}
    \end{figure}
\end{example}

\begin{example}\label{ex:car_cavity}
    (Car cavity problem).
   As our final example, we consider the sparse acoustic car-cavity model
    with porous seats studied in~\cite{limp:2022}. Its dimension is $n=15036$,
    and we use the four-term split
    \begin{equation}\nonumber
        T(x)=K_0+h_K(x)K_1
        -x^2M_0-x^2h_M(x)M_1,
    \end{equation}
    where $K_0,K_1,M_0,M_1\in\mathbb{R}^{15036\times15036}$ are sparse
    symmetric positive semidefinite matrices. The porous-material functions
    $h_K$ and $h_M$, including their physical parameters and principal
    square-root branches, are taken from~\cite{limp:2022}; removable values at
    zero are imposed explicitly.
\end{example}

    Following the smaller rectangular target in~\cite{limp:2022}, we compute
    Ritz values in the domain
    \[
        \Omega_{\rm car}=[1,300]+{\rm i}[0,510],
        \qquad \Gamma_{\rm car}:=\partial\Omega_{\rm car}.
    \]
    The approximation set contains $400$ arclength-uniform boundary points, and
    \textsf{m-d-Lawson} is applied to
       $ \bt(x)=
        [1,h_K(x),-x^2,-x^2h_M(x)]^{\T},$
    with   a type $(9,9)$ approximant. The error is
    \[
        \sqrt{e(\boldsymbol{\xi}^*)}
        =\max_{x\in{\cal X}}
        \|\bt(x)-\boldsymbol{\xi}^*(x)\|_2
        =4.4208\times10^{-9},
    \]
    and the approximation has seven finite poles, none in
    $\bar{\Omega}_{\rm car}$.

We solve the resulting linearization of size
    $9\times 15036=135324$ by CORK   with maximum subspace dimension $50$,
    $35$ retained Ritz values, target shift $150.5$, stopping tolerance
    $10^{-9}$, and   ten cyclic real shifts
  $
        \mu_j=1+\frac{299(j-1)}{9},~ j=1,\ldots,10.
    $
    Finite Ritz values are selected geometrically in $\Omega_{\rm car}$ and
    deduplicated with relative threshold $10^{-8}$; residuals are reported, but
    not used as an acceptance filter. 
    CORK in \textsf{NEP\_MiniMax-CORK} terminates after $63$ Krylov steps and returns $20$ distinct Ritz pairs
    in $\Omega_{\rm car}$, with largest and median residuals $\varrho_{\rm rel}(\lambda,\boldsymbol{u})$ in \eqref{eq:sandwich_residual} 
    {$5.3299\times10^{-11}$} and $1.5805\times10^{-12}$, respectively.
    Figure~\ref{fig:car_cavity_spectrum} shows the computed values, sampling
    boundary, CORK shifts, and nearby   poles.   For comparison, we also report results from SV-AAA with
    type $(d,d)$ rational approximants in Table \ref{table:car_residual_comparison}, where  the quality
    of the computed eigenpairs   is measured by the minimum, median and maximum values of $\varrho_{\rm rel}(\lambda,\pmb{u})$ defined in \eqref{eq:sandwich_residual}. 
 While we do not claim that the 20 reported Ritz values form the complete spectrum in $\Omega_{\rm car}$, the experiment demonstrates that  \textsf{NEP\_MiniMax-CORK}   is    effective for this sparse engineering NEP.

    \begin{figure}[t]
        \centering
        \includegraphics[width=0.95\textwidth]{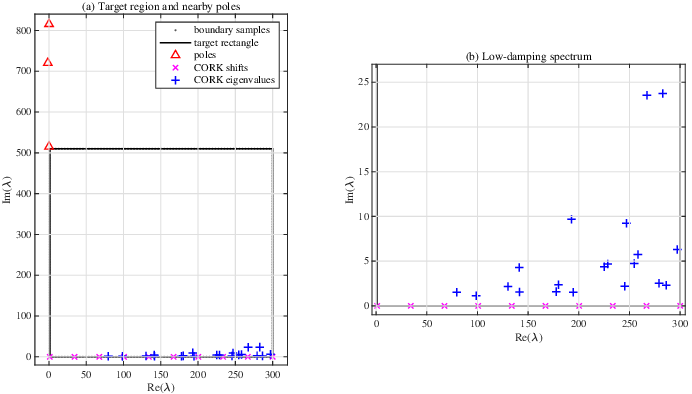}
        \caption{\footnotesize Numerical results for Example \ref{ex:car_cavity}  with the type $(9,9)$ approximation in 
        \textsf{NEP\_MiniMax-CORK}. Left: the samples, target rectangle, nearby finite
          poles, ten real CORK shifts, and 20 computed Ritz values.
        Right: an enlargement of the low-damping part of the target rectangle.}
        \label{fig:car_cavity_spectrum}
    \end{figure}

    \begin{table}[htbp]
    \caption{$\varrho_{\rm rel}(\lambda,\pmb{u})$
    of {\sf NEP\_MiniMax-CORK} and SV-AAA for Example \ref{ex:car_cavity}}
    \centering
    \begingroup
    \small
    \setlength{\tabcolsep}{6pt}
    \renewcommand{\arraystretch}{1.3}
    \begin{tabular}{|c|l|c|c|c|}
    \hline
      & Method & type $(8,8)$ & type $(9,9)$ & type $(10,10)$ \\
    \hline
    \multirow{2}{*}{minimum of $\varrho_{\rm rel}(\lambda,\pmb{u})$}
    & {\sf NEP\_MiniMax-CORK} & 7.9915e-11 & 1.7624e-13 & 1.2358e-13 \\
    \cline{2-5}
    & SV-AAA & 4.1732e-13 & 2.7413e-14 & 4.2523e-14 \\
    \hline
    \multirow{2}{*}{median of $\varrho_{\rm rel}(\lambda,\pmb{u})$}
    & {\sf NEP\_MiniMax-CORK} & 1.1660e-10 & 1.5805e-12 & 3.7550e-12 \\
    \cline{2-5}
    & SV-AAA & 8.3111e-10 & 1.9764e-10 & 2.4281e-10 \\
    \hline
    \multirow{2}{*}{maximum of $\varrho_{\rm rel}(\lambda,\pmb{u})$}
    & {\sf NEP\_MiniMax-CORK} & 2.1131e-10 & 5.3299e-11 & 3.2096e-11 \\
    \cline{2-5}
    & SV-AAA & 7.4388e-09 & 2.4722e-09 & 1.4986e-09 \\
    \hline
    \end{tabular}
    \endgroup
    \label{table:car_residual_comparison}
\end{table}


\section{Conclusion}\label{sec:conclusion}
{In this paper, we have presented {\sf NEP\_MiniMax}, an approach for nonlinear eigenvalue problems whose defining approximation step is rational minimax approximation with a shared denominator of the split coefficient vector via the {\tt m-d-Lawson} algorithm. This coefficient-vector approximant induces a rational matrix surrogate, which is solved through a strong linearization framework adaptable to V+A polynomial bases and integrated with compact rational Krylov (CORK) methods or rational-filter subspace iteration. The error analysis transfers the coefficient-vector minimax accuracy to a matrix approximation and eigenpair residual bound. Numerical experiments demonstrate competitive residual norms compared to Beyn's method, NLEIGS and SV-AAA at comparable approximation degrees. The main feature of the method is therefore the inexpensive global approximation of the split coefficient vector, together with structured eigensolution of the induced matrix surrogate.}

\section*{Appendix}\label{sec:appendix}
In this appendix, based   on the following theorem, we   give the proof of Proposition \ref{prop:stronglinear}. 
\begin{theorem}{\rm (\cite[Theorem 2.2]{amcl:2009})}
Let \(P(\lambda)\)\ be an \(n \times n\)\ regular matrix polynomial  {of grade $\gamma$, with $P_{\gamma}$ denoting the coefficient of $\lambda^\gamma$ in its monomial representation} (possibly zero) and let \(\lambda A-B\)\  be a \(\gamma n \times \gamma n\)\ linear matrix function. Assume that, for each distinct finite eigenvalue \(\lambda_{j}\)\ there exist functions \(D_{j}(\lambda)\)\ and \(G_{j}(\lambda)\) which are unimodular and analytic on a neighbourhood of \(\lambda_{j}\)\ and for which 
\begin{equation}
    \begin{bmatrix}P(\lambda)&\\ &I_{(\gamma-1)n}\end{bmatrix}=D_{j}(\lambda)(\lambda A-B)G_{j}(\lambda).
    \label{eq:local analytic of p}
\end{equation}
If \({P_{\gamma}}\) is singular (or zero), assume also that there are functions \(D_{0}(\lambda)\)\ and \(G_{0}(\lambda)\) which are unimodular and analytic on a neighbourhood of \(\lambda=0\)\ and for which 
\begin{equation}
    \begin{bmatrix}P^{\#}(\lambda)&\\ &I_{(\gamma-1)n}\end{bmatrix}=D_{0}(\lambda)(A-\lambda B)G_{0}(\lambda).
\end{equation}
Then \(\lambda A-B\)\ is a strong linearization of \(P(\lambda)\).
\label{thm:strong_reference}
\end{theorem}

{\it Proof of Proposition \ref{prop:stronglinear}.} The proof follows by analogous arguments to those in \cite[Theorem 3.1]{amcl:2009}. 
   First, we pre-multiply the pencil $\lambda C_{1}-C_{0}$ by the $n\gamma \times n\gamma$ block permutation matrix $S$ defined in \eqref{eq:S_permutation}. Since strong linearization is preserved under such equivalence transformations, $(\lambda C_{1}-C_{0})S$ remains a strong linearization if and only if $\lambda C_{1}-C_{0}$ is.
By Lemma \ref{lem:lu}(i), we obtain  
$(\lambda C_1-C_0)S = L(\lambda)U(\lambda)$
with $L(\lambda)$ and $U(\lambda)$ given in \eqref{eq:lu1_L_def} and \eqref{eq:lu1_U_def} respectively. Here, $U(\lambda)$ is   unimodular ($\det U(\lambda) \equiv 1$) and thus invertible.
Defining $\widetilde{L}(\lambda)$ as a modified version of $L(\lambda)$ where the final block entry is replaced by
$$\widetilde{L}_{\gamma,\gamma}(\lambda)=\frac{1}{\prod_{i=1}^{\gamma -1} h_{i+1,i}}I_{n},$$
we observe that $\det(\widetilde{L}(\lambda)) \equiv \pm 1$. This yields the refined decomposition
$$(\lambda C_{1}-C_{0})S = \widetilde{L}(\lambda) \begin{bmatrix}
I_{(\gamma-1)n}&\\
&P^*(\lambda)
\end{bmatrix} U(\lambda).$$
    Consequently, we obtain 
$\begin{bmatrix}
    I_{(\gamma-1)n} & \\
    & P^*(\lambda)
\end{bmatrix}
= D(\lambda)(\lambda C_1 - C_0) G(\lambda),$
where
$D(\lambda) = \widetilde{L}^{-1}(\lambda)$ and $G(\lambda) = S U^{-1}(\lambda)$ 
are both analytic and invertible at all finite eigenvalues of $P^*(\lambda)$. By Theorem \ref{thm:strong_reference}, $\lambda C_1 - C_0$ is thus a weak linearization of $P^*(\lambda)$.

For the strong linearization, we next examine its behavior under reversal by considering the reverse polynomial:
$$P^\#(\lambda) := (P^*(\lambda))^\# = \lambda^{\gamma} P^*(\lambda^{-1}).  $$
Following Theorem \ref{thm:strong_reference}, we must construct matrix functions $H(\lambda)$ and $K(\lambda)$ that are analytic and nonsingular near $\lambda = 0$ and satisfy the decomposition
\begin{equation}
\begin{bmatrix}
I_{(\gamma-1)n} & \\
& P^\#(\lambda)
\end{bmatrix}
= H(\lambda)(C_1 - \lambda C_0) K(\lambda).
\label{eq:rr3} 
\end{equation}

Step 1 (Reverse Factorization):
By  {Lemma \ref{lem:lu}(ii)}, the pencil admits a factorization
$\lambda C_1 - C_0 = L^\#(\lambda) U^\#(\lambda),$
where $L^\#(\lambda)$ and $U^\#(\lambda)$ are given in \eqref{eq:lu2_L_def} and \eqref{eq:lu2_U_def}. Applying the substitution $\lambda \mapsto \lambda^{-1}$, we derive
$C_1 - \lambda C_0 = \lambda L^\#(\lambda^{-1}) U^\#(\lambda^{-1}),$
which yields the reverse LU-factors:
$$L_1(\lambda) = \lambda L^\#(\lambda^{-1}), \quad U_1(\lambda) = U^\#(\lambda^{-1}).  $$

{
Step 2 (Local Analysis at $\lambda = 0$):
We next verify that the factors needed in \eqref{eq:rr3} are analytic and nonsingular at the origin. Since $\vartheta_j$ has exact degree $j$ and leading coefficient $k_j\ne0$, we have, as $\lambda\to0$,
\begin{equation}\label{eq:vartheta_ratio_origin}
\frac{\vartheta_{j-1}(\lambda^{-1})}{\vartheta_j(\lambda^{-1})}
=\frac{k_{j-1}}{k_j}\lambda+O(\lambda^2)
=h_{j+1,j}\lambda+O(\lambda^2),
\qquad j=1,\ldots,\gamma-1.
\end{equation}
Thus the entries of $U_1(\lambda)=U^\#(\lambda^{-1})$ have removable singularities at $\lambda=0$, and $U_1(0)=I_{\gamma n}$. Hence
$K(\lambda):=U_1^{-1}(\lambda)$
is analytic and nonsingular in a neighbourhood of $\lambda=0$.

Let
\begin{equation}\label{eq:alpha_origin}
\alpha(\lambda):=
\frac{\vartheta_0(\lambda^{-1})}
{\lambda^{\gamma-1}\vartheta_{\gamma-1}(\lambda^{-1})
\prod_{k=1}^{\gamma-1}h_{k+1,k}} .
\end{equation}
By the same leading-term argument, $\alpha(\lambda)$ also has a removable singularity at $\lambda=0$ and $\alpha(0)=1$. Define $\widetilde L_1(\lambda)$ by replacing the $(\gamma,\gamma)$ block of $L_1(\lambda)$ with
$\alpha(\lambda) I_n.$
Using \eqref{eq:uuu_def}, the original $(\gamma,\gamma)$ block of $L_1(\lambda)$ is
\begin{equation}\nonumber
\lambda
\frac{\vartheta_0(\lambda^{-1})}
{\vartheta_{\gamma-1}(\lambda^{-1})
\prod_{k=1}^{\gamma-1}h_{k+1,k}}
P^*(\lambda^{-1})
=\alpha(\lambda)P^\#(\lambda).
\end{equation}
Consequently,
\begin{equation}\label{eq:L1_reversal_split}
L_1(\lambda)=\widetilde L_1(\lambda)
\begin{bmatrix}
I_{(\gamma-1)n}&\\
&P^\#(\lambda)
\end{bmatrix}.
\end{equation}
Moreover, $\widetilde L_1(\lambda)$ is block lower triangular. Its first $\gamma-1$ diagonal blocks are
\[
\lambda h_{j+1,j}
\frac{\vartheta_j(\lambda^{-1})}{\vartheta_{j-1}(\lambda^{-1})}I_n,
\qquad j=1,\ldots,\gamma-1,
\]
and its last diagonal block is $\alpha(\lambda)I_n$. Therefore,
\[
\prod_{j=1}^{\gamma-1}
\lambda h_{j+1,j}
\frac{\vartheta_j(\lambda^{-1})}{\vartheta_{j-1}(\lambda^{-1})}
\cdot \alpha(\lambda)=1,
\]
which implies $\det\widetilde L_1(\lambda)\equiv1$. Hence
$H(\lambda):=\widetilde L_1^{-1}(\lambda)$
is analytic and nonsingular near $\lambda=0$. Combining
$C_1-\lambda C_0=L_1(\lambda)U_1(\lambda)$ with
\eqref{eq:L1_reversal_split} gives
$
\begin{bmatrix}
I_{(\gamma-1)n}&\\
&P^\#(\lambda)
\end{bmatrix}
=H(\lambda)(C_1-\lambda C_0)K(\lambda),
$
which is precisely \eqref{eq:rr3}. This completes the proof. \hfill $\square$
}

{\small 
\def\noopsort#1{}\def\l{\char32l}\def\v#1{{\accent20 #1}}
  \let\^^_=\v\def\hbk{hardback}\def\pbk{paperback}
\providecommand{\href}[2]{#2}
\providecommand{\arxiv}[1]{\href{http://arxiv.org/abs/#1}{arXiv:#1}}
\providecommand{\url}[1]{\texttt{#1}}
\providecommand{\urlprefix}{URL }

}


\end{document}